\documentclass[11pt,leqno]{amsart}
\usepackage{graphicx} % Required for inserting images
\usepackage{amsmath,amsfonts,amsthm,mathrsfs,amssymb,amscd,comment,enumitem,amsxtra,url,tikz-cd}
\input{xypic}
\xyoption{all}

\usepackage{xcolor} % Required for specifying customized colours 
\colorlet{mdtRed}{red!50!black}
\definecolor{dblue}{rgb}{0,0,.6}
\usepackage[colorlinks]{hyperref}
\hypersetup{linkcolor=blue,citecolor=dblue,filecolor=dullmagenta,urlcolor=mdtRed}

\numberwithin{equation}{section}
\newtheorem{theorem}[equation]{Theorem}
\newtheorem{corollary}[equation]{Corollary}
\newtheorem{lemma}[equation]{Lemma}
\newtheorem{remark}[equation]{Remark}
\newtheorem{proposition}[equation]{Proposition}
\newtheorem{definition}[equation]{Definition}
\newtheorem{conjecture}[equation]{Conjecture}

\newtheorem*{theorem*}{Theorem}
\newtheorem*{corollary*}{Corollary}
\newtheorem*{proposition*}{Proposition}

\newcommand{\R}{\mathbb{R}}

\newcommand{\mf}[1]{\mathfrak{#1}}

\newcommand{\mb}[1]{\mathbb{#1}}
\newcommand{\mc}[1]{\mathcal{#1}}
\renewcommand{\t}[1]{\tilde{#1}}

\begin{document}
	
	\title{Seshadri constants and negative curves on blowups of ruled surfaces}
	%Seshadri constants on blowups of ruled surfaces\\
	%OR\\}

\author[C. J. Jacob]{Cyril J. Jacob}\address{Chennai Mathematical Institute, H1 SIPCOT IT Park, Siruseri, Kelambakkam 603103, India}
\email{cyril@cmi.ac.in}

\author[B. Khan]{Bivas Khan}
\address{Chennai Mathematical Institute, H1 SIPCOT IT Park, Siruseri, Kelambakkam 603103, India}
\email{bivaskhan10@gmail.com}

\author[R. Sebastian]{Ronnie Sebastian} 
\address{Department of Mathematics, Indian Institute of Technology Bombay, Powai, Mumbai 
	400076, Maharashtra, India}
\email{ronnie@iitb.ac.in} 


\begin{abstract}
	In this article we compute Seshadri constants of ample 
	line bundles on the blowup of Hirzebruch surface 
	$\mathbb{F}_e$ at $r\leqslant e+3$ very general points. 
	Similarly, we compute Seshadri constants on the blowups of 
	certain decomposable ruled surfaces over smooth curves of 
	non-zero genus. We also prove some results related to bounded negativity 
	of blowups of Hirzebruch surfaces and ruled surfaces. 	
\end{abstract}
\maketitle

\section{Introduction}
Seshadri constants measure the local positivity of an 
ample line bundle on an algebraic variety. 
These were introduced by Demailly  in order 
to study Fujita conjecture \cite{Dem}. 
The following definition of Seshadri constants
is motivated by the ampleness criterion for line 
bundles given by Seshadri (see \cite[Theorem 7.1]{Ha70}). 

\begin{definition}
	Let $X$ be a smooth projective variety. Let $L$ 
	be a nef line bundle on $X$. 
	We define the \textit{Seshadri constant of $L$ at a point $x\in X$} 
	as
	$$\varepsilon(X,L,x):=\inf \left\{\frac{L\cdot C}{{\rm mult}_x C}~\Big| ~C \text{ is an integral curve on $X$ passing through $x$}\right\}.$$
\end{definition}

The following theorem gives an equivalent definition 
for the Seshadri constant using the nefness of certain 
line bundles on the blowup of $X$ at $x$. 
\begin{theorem} \cite[Proposition 5.1.5]{Laz2004} \label{EquivDefn}
	Let $X$ be a smooth projective variety and 
	$L$ be a nef line bundle on $X$. For a point $x\in X$, 
	let $\pi: X_x \to X$ denote the blowup of $X$ 
	at $x$ and let $E:=\pi^{-1}(\{x\})$ 
	denote the exceptional divisor. Then, we have
	$$\varepsilon(X,L,x)=\sup \{s\in \mathbb{R}_{\geqslant 0}~|~\pi^\ast L-sE \text{ is nef line bundle on $X_x$}\}.$$
\end{theorem}

It is easy to see that $\pi^\ast L-\varepsilon(X,L,x)E$ 
is a nef line bundle on $X_x$ and hence $\varepsilon(X,L,x)\leqslant \sqrt{L^2}$.
More details about Seshadri constants and local 
positivity can be found in \cite[Section 5]{Laz2004} and \cite{primer}.
%Study of Seshadri constants lead to results about the global generation, very ampleness or $s$-jets generation of adjoint linear systems (see \cite{EKL95}). It is also related to Castelnuovo-Mumford regularity, and the existence of cscK metrices (see \cite{primer}).  

Since the definition of Seshadri constants depends 
on individual curves rather than linear or numerical 
equivalence classes, it is not easy to calculate Seshadri 
constants on varieties, even on surfaces. Explicit 
computation of Seshadri constants are given for several surfaces, for example 
simple abelian surfaces (\cite{Bauer98}), rational 
and elliptical ruled surfaces (\cite{Garcia06}), 
hyperelliptic surfaces (\cite{Han2018}), rational 
surfaces $S$ with $\text{dim} |-K_S| \geqslant 1$ (\cite{Sano14}) etc. 
For general bounds for Seshadri constants 
see \cite{Bauer99}, \cite{Naka05}, \cite{Steffens98}.

Consider the \textit{Hirzebruch surfaces} 
$\mathbb{F}_e=\mathbb{P}(\mathcal{O}_{\mathbb{P}^1}\oplus \mathcal{O}_{\mathbb{P}^1}(-e))$, 
the projectivization of the rank $2$ vector bundle 
$\mathcal{O}_{\mathbb{P}^1}\oplus \mathcal{O}_{\mathbb{P}^1}(-e)$ on $\mathbb{P}^1$ for $e\in \mathbb{Z}_{\geqslant 0}$.
Let $C_e$ denote a curve given by the vanishing 
of a global section of the line 
bundle $\mathcal{O}_{\mathbb{F}_e}(1)$ (this curve
is unique if $e>0$). Then $C_e^2=-e$. Let $f_e$ denote 
a fiber of the natural map $\pi: \mathbb{F}_e \to \mathbb{P}^1$. 
%Since all points in $\mathbb{P}^1$ are linearly equivalent to 
%each other, we can see that all fibers of the map 
%$\pi$ are linearly equivalent. 
With these we can see that
$$\text{Pic }(\mathbb{F}_e) = \mathbb{Z}[C_e]\oplus\mathbb{Z}[f_e]$$ 
and intersection products on $\mathbb{F}_e$ are given 
by $C_e^2=-e,f_e^2=0$ and $C_e\cdot f_e=1$.

It is know that the minimal rational surfaces are the 
projective plane and the Hirzebruch surfaces 
$\mathbb{F}_e \, (\text{with } e \neq 1)$, that is,
any rational surface can be obtained as a blowup of 
these surfaces at finitely many points (\cite[Theorem V.10]{Beauville}). 
So in order to study the Seshadri constants on rational 
surfaces it is enough to study the blowups of these surfaces. 
The case of blowups of $\mathbb{P}^2$ has been 
extensively studied (see \cite{DKMZ16}, \cite{FHHSS20}, \cite{HH18}, \cite{Sano14}). 

In this article we focus on the blowup of Hirzebruch 
surfaces and more generally of ruled surfaces. 
Let $e\in \mb Z_{\geqslant 0}$ and let $r$ be an integer.
Let $P=\{p_1,\dots,p_r\}$ be a set of $r$ distinct points on $\mathbb{F}_e$. 
Let $\pi_P: \mathbb{F}_{e,P} \to \mathbb{F}_e$ be the 
blowup of $\mathbb{F}_e$ at the points in $P$. 
Let $E_i$ denote the exceptional divisor corresponding 
to the point $p_i\in P$. Then we have, 
$$\text{Pic }(\mathbb{F}_{e,P})=\mathbb{Z}[C_e]\oplus\mathbb{Z}[f_e]\oplus \mathbb{Z}[E_1]\oplus \dots \oplus \mathbb{Z}[E_r]\,.$$
In \cite[Theorem 3.6 and Theorem 3.8]{HJNS24}, the 
authors have computed Seshadri constants of 
an ample line bundle on the 
blowup of $\mathbb{F}_e$ at 
$r\leqslant e+1$ very general points.
An important ingredient which they use is that 
these surfaces have effective anticanonical bundle. However, this is true for 
$\mb F_e$ blown up at $r\leqslant e+5$ points 
\cite[Proposition 3]{FL2021}. 
In this article we extend the computation of 
Seshadri constants to the case $r \leqslant e+3$. 
For $r=e+2$ we have the following. 
\begin{theorem} [see Theorem \ref{SC e+2}]
	For $r=e+2$, let $P= \{p_1,\dots,p_r\}\subset \mathbb{F}_e$ be a set 
	of $r$ distinct points. Let $x \in \mathbb{F}_{e,P}$ be such 
	that $P'=P\cup \{\pi_{P}(x)\}$ be a set of $r+1$ points 
	which are in very general position.	
	Let 
	$L=\alpha C_e+\beta f_e-\sum_{i=1}^r\mu_iE_i$ be an ample line 
	bundle on $\mathbb{F}_{e,P}$. The
	Seshadri constant of $L$ at $x$ is
	$$\varepsilon(\mathbb{F}_{e,P},L,x)=
	\min \left( \alpha,\,\, \beta-\sum_{j=1}^e \mu_{i_j}, \,\,\beta+\alpha-\sum_{i=1}^r\mu_i\right),$$
	where the sum in $\beta-\sum_{j=1}^e \mu_{i_j}$ is taken over the largest $e$ $\mu_i$.
	When $e=0$ this sum is to be taken as $\beta$. 
\end{theorem}
For $r=e+3$ we have the following. 
See \eqref{def A(L)} for definition of $A(L)$ and $B(L)$ in the case $r=e+3$.
\begin{theorem}[see Theorem \ref{SC e+3}]
	Let $r=e+3$. With notation as above, the
	Seshadri constant of $L$ at $x$ is
	\begin{align*}
		\varepsilon(\mathbb{F}_{e,P},L,x)&=\min \left( \alpha,\,\, \beta,\,\, \beta+\alpha-\sum_{j=1}^{2}\mu_{i_j}\right)\,,\qquad \qquad \text{for $e=0$}\,,\\
		\varepsilon(\mathbb{F}_{e,P},L,x)&=\min \left( \alpha,\,\,\beta-\sum_{j=1}^e \mu_{i_j}, \,\, \beta+\alpha-\sum_{i=1}^{e+2}\mu_{i_j},\,\, A(L), \,\,B(L)\right)\,,\qquad \qquad \text{for $e>0$}\,,
	\end{align*}
	where the sum in $\beta+\alpha-\sum_{j=1}^{2}\mu_{i_j}$ is over the two largest $\mu_i$, 
	the sum in $\beta+\alpha-\sum_{j=1}^{e}\mu_{i_j}$ is over the largest $e$ $\mu_i$,
	and the sum in $\beta+\alpha-\sum_{j=1}^{e+2}\mu_{i_j}$ is over the largest $e+2$ $\mu_i$.
\end{theorem}

We also compute 
Seshadri constants on the blowup of
certain ruled surfaces.
Let $\Gamma$ be a smooth curve of genus $g \geqslant 1$. 
Consider the ruled surface 
$$\pi: X=\mathbb{P}(\mathcal{O}_{\Gamma}\oplus \mathcal{L}) 
\to \Gamma\,,$$
where $\mc L$ is a line bundle on $\Gamma$ 
with ${\rm deg}(\mc L)=:-e<0$. 
Let $P=\{p_1,\dots,p_r\}\subset X$, where $r<e+2-3g$. Let 
$\pi_P: X_P\to X$ be the blowup of $X$ at the 
points in  $P$. We compute Seshadri constants at points on $X_P$ in 
Theorem \ref{SC for ruled sur}. Here again we use the fact that 
these surfaces have an effective anticanonical divisor. 
The strategy of the proofs of the above Theorems is explained in detail in 
Section \ref{preliminaries section}, before Proposition \ref{SC Computation}.

In \cite[Conjecture 4.8]{HJNS24}, the authors have proposed the following conjecture:
\begin{conjecture}\label{ConjNeg}
	Let $P=\{p_1,\dots,p_r\}\subset \mb F_e$ be $r$ 
	points in very general  position. Let $\pi_P: \mathbb{F}_{e,P} \to \mathbb{F}_e$ 
	be the blowup of $\mathbb{F}_e$ at points in $P$. If $C$ is an integral curve on 
	$\mathbb{F}_{e,P}$ such that $C^2<0$. Then $C$ is either the strict transform of 
	$C_e$ or a $(-1)$-curve.
\end{conjecture}
In \cite{HJNS24-1}, 
it is shown that this conjecture implies the existence of a 
line bundle with irrational Seshadri constants.
The above conjecture is analogous to the well known $(-1)$-curves conjecture or 
Weak SHGH Conjecture for $\mathbb{P}^2$. The authors in  \cite{HJNS24}
proved that the Conjecture \ref{ConjNeg} is true when $r \leqslant e+2$. 
We have improved it for the cases $r=e+3$ and $r = e+4$.
\begin{theorem}[Theorem \ref{Conj 4.8 HJNS}]
	Let $P=\{p_1,\dots,p_r\}\subset \mb F_e$ be $r \leqslant e+4$ 
	points in very general  position. Recall $\pi_P: \mathbb{F}_{e,P} \to \mathbb{F}_e$ 
	is the blowup of $\mathbb{F}_e$ at points in $P$. Let $C$ be an integral curve on 
	$\mathbb{F}_{e,P}$ such that $C^2<0$. Then $C$ is either the strict transform of 
	$C_e$ or a $(-1)$-curve.
\end{theorem}
In \cite[Conjecture 2.1]{HJNS24-1}, Conjecture \ref{ConjNeg} is 
extended to the blowup of any ruled surface. Let 
$\phi:X=\mb P(E)\to \Gamma$ be a ruled surface over a smooth 
curve $\Gamma$ of genus $g$,
with invariant $e>0$.
Let $P=\{p_1,\dots,p_r\}\subset X$ be $r$ distinct 
points in $X$. Let $\pi_P: X_P \to X$ 
be the blowup of $X$ at the points in $P$.
It is rather easy to see that \cite[Conjecture 2.1]{HJNS24-1} 
holds whenever $r \leqslant e$, see Corollary \ref{cor conj 2.1}.

In the final section, we have provided bounds on the self-intersections 
of curves on ruled surfaces. This is related to the celebrated Bounded 
Negativity Conjecture. We say that a surface $X$ has bounded negativity if there exists a constant $b(X)$ such that $C^2\geqslant- b(X)$ for all integral curves $C$ on $X$. In positive characteristic this conjecture 
is false, see \cite[Remark I.2.2]{Harb-notes-global}. In \cite[Conjecture 3.7.1]{BBC12}, the authors 
formulated a variant of the BNC, known as the weighted bounded negavity 
conjecture. Here the lower bound on $C^2$ is allowed to depend on 
the degree of $C$ with respect to any nef and big divisor on the surface. 
See \cite[Theorem B]{D-LP20}, \cite[Section 4]{D-LP20}, 
\cite{D-CF24} for some results related to weighted bounded 
negativity conjecture on blowups of 
$\mathbb{P}^2$, and some partial 
results for blowups of Hirzebruch surface. 
In the same spirit, we have obtained bounds for blowup of 
Hirzebruch surfaces. 
\begin{theorem}[Theorem \ref{wbnc-hirz}]
	Let $P=\{p_1,\dots,p_r\}\subset \mb F_e$ be $r$ 
	points. Recall that $\pi_P: \mathbb{F}_{e,P} \to \mathbb{F}_e$ 
	is the blowup of $\mathbb{F}_e$ at points in $P$. Let $C$ be an integral curve on 
	$\mathbb{F}_{e,P}$ such that $C^2<0$. Then either $C$ is an exceptional divisor or 
	$$C^2\geqslant {\rm min}\{-2,-e-r \}+ 
	\left(e+2-\left\lfloor \frac{r+e}{2}\right\rfloor\right)(C.f_e)\,.$$
\end{theorem}
More generally for any ruled surface we show the following.
\begin{theorem}[Theorem \ref{wbnc-ruled}]
	Let $\phi:X=\mb P(E)\to \Gamma$ be a ruled surface over a smooth curve $\Gamma$ of genus $g$,
	with invariant $e$.
	Let $P=\{p_1,\dots,p_r\}\subset X$ be $r$ distinct 
	points in $X$. Let $\pi_P: X_P \to X$ 
	be the blowup of $X$ at the points in $P$. Let $C$ be an integral curve on 
	$X_P$ such that $C^2<0$. Let
	$$\lambda:={\rm max}\left\{2g-1,2g-1+e,g+\left\lfloor \frac{r+e}{2}\right\rfloor\right\}\,.$$ 
	Then either $C$ is an exceptional divisor or 
	$$C^2\geqslant {\rm min}\{-2,-r \}+ 
	\left(e+2-\lambda-2g\right)(C.f_e)\,.$$
\end{theorem}

\noindent{\bf Notation}. We work over an 
algebraically closed field of characteristic zero. Throughout this article, to avoid notation from becoming 
overly cumbersome, we shall use the following convention. Let $X$ be a smooth projective
surface. Let $P=\{p_1,\ldots,p_r\}$ be $r$ points on $X$ and let $f:\widetilde{X}\to X$ 
denote the blowup of $X$ at $P$. If $C\subset X$ is a curve, then we shall abuse 
notation and denote $f^*C$ by $C$. The strict transform of $C$ will be denoted 
by $\widetilde{C}$. By abuse of notation, for $x\notin P$ we use $x$ 
also for $\pi_P^{-1}(\{x\})$.
We will say that a property holds at a general point of a 
variety \(X\) if it holds for a non-empty
Zariski-open subset of X. We will say that a property holds 
at a very general point if it is 
satisfied by a subset whose complement is the union of countably 
many proper closed subvarieties of $X$. We will freely switch between line bundles and divisors.

\subsection*{Acknowledgments} We thank Krishna Hanumanthu for useful discussions.
The first two authors thank IIT Bombay for its hospitality. 
They are also partially supported by a grant from the Infosys Foundation.

\section{Preliminaries}\label{preliminaries section}
Let $X$ be a smooth projective surface. 
The following result is well known, but we include a proof 
for the benefit of the authors and the readers. Let $X^r$ denote the 
$r$-fold product of $X$ with itself.

\begin{proposition}\label{general position}
	Let the Picard group of $X$ be countable. 
	There is a subset $T\subset X^r$ which has the following two properties:
	\begin{itemize}
		\item It is the complement of a countable union of proper closed subsets of $X^r$.
		\item Let $(p_1,\ldots, p_r)\in T$. If $D$ is a curve
		which contains all these points, then $r<h^0(X, D)$. 
	\end{itemize}
\end{proposition}
\begin{proof}
	Let us fix the integer $r\geqslant 1$.  Let $L$ be a line
	bundle on $X$ such that $0<h^0(X, L)\leqslant r$. We begin with defining a closed subset 
	$Z_L\subset X^r$. Let $q_i:X^r\to X$ denote projection to the $i$th factor. 
	Consider the map of sheaves on $X^r$
	$$\psi:H^0(X,L)\otimes \mc O_{X^r}\to \bigoplus_{i=1}^rq_i^*L\,,$$
	which is defined as follows. The map $H^0(X,L)\otimes \mc O_{X^r}\to q_i^*L$
	is obtained by pulling back the canonical map $H^0(X,L)\otimes \mc O_{X}\to L$
	on $X$ along $q_i$. If $s\in H^0(X,L)$ is a section, then the restriction of the above
	map over the point $(p_1,\ldots,p_r)\in X^r$ is given by 
	$$s\mapsto (s\vert_{p_1},\ldots, s\vert_{p_r})\,.$$
	
	Recall the following basic fact. 
	Let $Y$ be an integral scheme, $\mc F$, $\mc G$ be locally free coherent sheaves on $Y$,
	and $\phi:\mc F\to \mc G$ be a morphism between these. Suppose there is a point $y\in Y$ such 
	that the restriction 
	$\phi\vert_y:\mc F\otimes \mc O_Y/{\mf m_y}\to \mc G \otimes \mc O_Y/{\mf m_y}$ 
	is an inclusion. Then the map $\phi$ is an inclusion and there is an open neighbourhood 
	$U$ of $y$ such that the cokernel is locally free on $U$. 
	
	If $h^0(X,L)=1$, then the morphism $\psi$ being nonzero is an inclusion. If $h^0(X,L)>1$,
	then there is an open subset $V\subset X$ and a morphism $V\to \mb P(H^0(X,L))$.  
	Let $(p_1,\ldots,p_r)\in V^r$ be such that the linear subspace generated by the images of $p_i$
	is all of $\mb P(H^0(X,L))$. Note that for this to happen it is necessary that 
	$h^0(X,L)\leqslant r$. We claim that the map
	$$s\mapsto (s\vert_{p_1},\ldots, s\vert_{p_r})\,$$
	is an inclusion. If not, there is a nonzero section of $L$ which vanishes at all these 
	points, that is, a hyperplane in $\mb P(H^0(X,L))$ which contains all the $p_i$,
	which is not possible. Thus, using the fact mentioned in the preceding para,
	it follows that the map $\psi$ is an inclusion. Further,
	there is an open subset $U_L\subset X^r$ such that the cokernel of $\psi$ is locally 
	free over $U_L$. If $P=(p_1,\ldots,p_r)\in U_L$, then clearly the map 
	$\psi\vert_{P}$ is an inclusion. Let $Z_L:=X^r\setminus U_L$. Let $Z$ be the union 
	of $Z_L$ over all line bundles $L$ with $0<h^0(X,L)\leqslant r$. 
	We take $T$ to be $X^r\setminus Z$. 
	
	We claim that $T$ has the required property. Let $P=(p_1,\ldots,p_r)\in T$. 
	Let $D$ be a curve which contains the 
	points $p_i$. Let $L=\mc O_X(D)$. Then $h^0(X,L)>0$. Assume that $h^0(X,L)\leqslant r$.
	The point $P=(p_1,\ldots,p_r)\in U_L$. As we observed above, the map $\psi\vert_P$ is an 
	inclusion. The curve $D$ is the vanishing locus of a nonzero section $s\in H^0(X,L)$
	which passes through the points $p_i$. This shows that $\psi\vert_P(s)=0$, 
	a contradiction. 
\end{proof}

Denote 
$NS(X)_\R:=NS(X)\otimes_{\mathbb Z} \R$,
where $NS(X)$ denotes the N\'eron-Severi group of $X$. 
For $m\in \{1,2\}$, a $(-m)$-curve will mean a smooth irreducible rational curve 
$C$ such that $C^2=-m$. 

\begin{proposition}\label{usefulprop-1}
	Let $X$ be a smooth projective surface such that the canonical divisor $K_X$
	satisfies $H^0(X,-K_X)\neq 0$. Let $L\in NS(X)_\R$, 
	be such that $L^2\geqslant 0$
	and $L.H> 0$ for some ample divisor $H$. 
	Assume that 
	\begin{enumerate}
		\item $L \cdot C\geqslant 0$ for all $(-1)$ and $(-2)$-curves,
		\item $L \cdot C\geqslant 0$ for every integral curve $C$ in the base 
		locus of $-K_X$ with $C^2<0$. 
	\end{enumerate}
	Then $L$ is nef. 
\end{proposition}
\begin{proof}
	To show that $L$ is nef, it suffices to show that for every integral curve $C$,
	we have $L \cdot C\geqslant 0$. Let $C$ be an integral curve such that $C^2\geqslant 0$. 
	If possible, let $L \cdot C<0$. Since $C^2\geqslant 0$, it follows that $C$ is nef. 
	Thus, $C+\alpha H$ is ample for $\alpha>0$. If we take $\alpha_0=-(L \cdot C)/(L \cdot H)>0$
	then we get $L \cdot (C+\alpha_0 H)=0$. Since $C+\alpha_0 H$ is ample, from the Hodge
	Index Theorem it follows that $L^2<0$ or $L=0$. This gives a contradiction. 
	Thus, if $C^2\geqslant 0$ then $L \cdot C\geqslant 0$. 
	
	Next, consider the case when $C$ is an integral curve with $C^2<0$. 
	If $C$ is in the base locus of $-K_X$, then by assumption we have $L\cdot C\geqslant 0$. 
	If $C$ is not in the base locus of $-K_X$, then $-K_X \cdot C\geqslant 0$, that is, 
	$K_X \cdot C\leqslant0$. The genus formula, $g(C)=1+C\cdot (C+K_X)/2$, shows that 
	the genus is forced 
	to be 0, that is, $C$ is a smooth rational curve. Moreover, it follows easily 
	that $C$ is either a $(-1)$-curve or a $(-2)$-curve. 
	By our assumption, it follows 
	that $L \cdot C\geqslant 0$. This completes the proof of the Proposition. 
\end{proof}

Let $x\in X$ be a point. Let $L$ 
be an ample line bundle on $X$. Let $\pi:\widetilde X\to X$ denote the blowup
of $X$ at the point $x$. Let $E_x$ denote the exceptional divisor. 
Recall that the Seshadri constant of $L$ at $x$ is defined as 
$$\varepsilon(X,L,x):={\rm sup}\{s\geqslant 0\,\vert\, \pi^*L-sE_x \text{ is nef }\}\,.$$
\begin{definition}\label{defDelta}
	Let $\Delta$ denote the set of numerical classes of integral curves 
	$\widetilde{C}\subset \widetilde{X}$, such that 
	$\widetilde{C}$ is the strict transform of an integral curve $C\subset X$
	which contains $x$, and $\widetilde{C}$ is:
	\begin{itemize}
		\item a $(-1)$-curve, or
		\item a $(-2)$-curve, or
		\item in the base locus of $-K_{\widetilde{X}}$ and $\widetilde{C}^2<0$.
	\end{itemize}
\end{definition}	
\begin{proposition}\label{SC computation-1}
	With the notation as above, assume that the canonical divisor $K_{\widetilde X}$
	satisfies $$H^0(\widetilde X,-K_{\widetilde X})\neq 0\,.$$ Then
	$$\varepsilon(X,L,x)=\sup\{s\geqslant 0\,|\,(\pi^\ast L-sE_x)\cdot \widetilde C 
	\geqslant 0 \text{ for all }\widetilde C\in \Delta\}\,.$$
\end{proposition}
\begin{proof}
	If we let $L'_s=\pi^*L-sE_x$, then if $L'_s$ is nef, we have $L'^2_s\geqslant 0$. 
	This shows that $\varepsilon(X,L,x)\leqslant \sqrt{L^2}$. Thus, we have 
	$$\varepsilon(X,L,x)={\rm sup}
	\{\sqrt{L^2} \geqslant s\geqslant 0\,\vert\, \pi^*L-sE_x \text{ is nef }\}\,.$$
	If $H$ is an ample divisor on $X$, then for $\alpha$ very large, we have $\alpha H-E_x$ 
	is ample on $\widetilde{X}$. For $0\leqslant s\leqslant \sqrt{L^2} $ and $\alpha$ very large, we have 
	$$(\pi^*L-sE_x)  \cdot (\alpha H-E_x)=\alpha L \cdot H-s\geqslant \alpha L \cdot H-\sqrt{L^2}>0\,.$$
	Thus, for every $0\leqslant s\leqslant \sqrt{L^2}$, the 
	element $L'_s\in NS(\widetilde{X})_\R$ satisfies the 
	conditions $L'^2_s\geqslant 0$ and $L'_s\cdot (\alpha H-E_x)>0$. 
	We claim that if $L'_s\cdot \widetilde C\geqslant 0$ for all $\widetilde C\in \Delta$,
	then $L'_s$ is nef. 
	As the anticanonical divisor $-K_{\widetilde{X}}$ of $\widetilde{X}$ 
	is effective, we may apply Proposition \ref{usefulprop-1}.
	If $\widetilde C\subset \widetilde{X}$ is such an integral curve, 
	and $x\notin \pi(\widetilde C)$, then clearly 
	$L'_s\cdot \widetilde C\geqslant 0$. If $\pi(\widetilde C)=\{x\}$ is a point, 
	then clearly $\widetilde C=E_x$ and $L'_s\cdot \widetilde C\geqslant 0$. 
	We easily conclude that $L'_s$ is nef iff $L'_s\cdot \widetilde C\geqslant 0$ 
	for all $\widetilde C\in \Delta$.
\end{proof}

This brings us to the strategy we employ in this article. 
The main idea is contained in \cite[Proposition 4.1]{Sano14}
and is also used in \cite{HJNS24}.
Let $S$ be a smooth 
projective surface and let $X$ be the blowup of $S$ at $r$ very general points. 
Let $x\in X$ be a very general point and let $\widetilde X$ be the blowup of $X$
at $x$. We will be interested in situations in which 
$H^0(\widetilde X,-K_{\widetilde{X}})\neq 0$. In the situations we consider, 
the set $\Delta$ will turn out 
to be a finite list. Then the Seshadri 
constant $\varepsilon(X,L,x)$ is the largest $s$ such that 
$(\pi^*L-sE_x).\widetilde C\geqslant 0$ for all $\widetilde C\in \Delta$.
In some situations we will need the following variant of Propostion \ref{SC computation-1}.

\begin{proposition}\label{SC Computation}
	Let $X$ and $\widetilde{X}$ be as above such that 
	$H^0(\widetilde{X},-K_{\widetilde{X}})\neq 0$.
	Let $\Lambda$ be a set of effective divisors in $\widetilde{X}$ such that 
	$\Delta\subset \Lambda$. Then for any ample line bundle $L$ on $X$,
	$$\varepsilon(X,L,x)=\sup\{s\geqslant 0\,|\,(\pi^\ast L-sE_x)\cdot C' 
	\geqslant 0 \text{ for all }C'\in \Lambda\}\,.$$
	
\end{proposition}
\begin{proof}
	From Proposition \ref{SC computation-1}, we have
	$$\varepsilon(X,L,x)=\sup\{s\geqslant 0\,|\,(\pi^\ast L-sE_x)\cdot \widetilde C
	\geqslant 0 \text{ for all }\widetilde C\in \Delta\}\,.$$
	Let 
	$$\varepsilon_1=\sup\{s\geqslant 0\,|\,(\pi^\ast L-sE_x)\cdot C' 
	\geqslant 0 \text{ for all }C'\in \Lambda\}\,.$$ 
	Our goal is to show $\varepsilon_1=\varepsilon(X,L,x)$.
	
	Since $\Delta \subset \Lambda$, it is easy to see that 
	$\varepsilon_1\leqslant\varepsilon(X,L,x)$. Now consider a class 
	$\pi^\ast L-sE_x$ such that  
	$(\pi^\ast L-sE_x)\cdot \widetilde C \geqslant 0$ for all $\widetilde C\in \Delta$. 
	By Proposition \ref{usefulprop-1}, 
	$\pi^\ast L-sE_x$ is nef, so for every 
	effective divisor $C'$ we have $(\pi^\ast L-sE_x)\cdot C'\geqslant 0$. 
	In particular we have 
	$(\pi^\ast L-sE_x)\cdot C' \geqslant 0 \text{ for all }C'\in \Lambda$. 
	Hence $\varepsilon(X,L,x)\leqslant \varepsilon_1$.
\end{proof}

Let $e\geqslant 0$ be an integer and let 
\begin{equation}\label{def-pi}
	\pi:\mathbb F_e\to \mb P^1
\end{equation}
denote the Hirzebruch surface $\mb F_e=\mb P(\mc O_{\mb P^1}\oplus \mc O_{\mb P^1}(-e))$.

\begin{lemma}\label{pfw_phi}
	For \(a \geqslant 0\), we have
	$$\pi_\ast(aC_e)=\mathcal{O}_{\mathbb{P}^1}\oplus\mathcal{O}_{\mathbb{P}^1}(-e)\oplus\mathcal{O}_{\mathbb{P}^1}(-2e)\oplus\dots \oplus \mathcal{O}_{\mathbb{P}^1}(-ae).$$
\end{lemma}
\begin{proof}
	From \cite[Chapter 2, Proposition 7.11]{Ha} we can see that 
	\begin{align*}
		\pi_\ast(aC_e)&={\rm Sym}^a(\mathcal{O}_{\mathbb{P}^1}\oplus\mathcal{O}_{\mathbb{P}^1}(-e))\\
		&=\bigoplus_{i+j=a}\left({\rm Sym}^i(\mathcal{O}_{\mathbb{P}^1})\otimes{\rm Sym}^j(\mathcal{O}_{\mathbb{P}^1}(-e))\right) \\
		&=\mathcal{O}_{\mathbb{P}^1}\oplus\mathcal{O}_{\mathbb{P}^1}(-e)\oplus\mathcal{O}_{\mathbb{P}^1}(-2e)\oplus\dots \oplus \mathcal{O}_{\mathbb{P}^1}(-ae).
	\end{align*}
\end{proof}

\begin{proposition}\label{dim}
	Let $ a \geqslant  0 \text{ and }b \geqslant ae$. Then we have
	$$h^0(\mb F_e,aC_e+bf_e)=(a+1)(b+1-\frac{ae}{2}).$$
\end{proposition}

\begin{proof}
	Clearly, $h^0(\mb F_e,aC_e+bf_e)=h^0(\mathbb{P}^1,\pi_\ast(aC_e+bf_e))$. 
	By the projection formula, we have
	\[h^0(\mb F_e,aC_e+bf_e)=h^0(\mathbb{P}^1,\pi_\ast(aC_e)\otimes \mathcal{O}_{\mathbb{P}^1}(b)).\]
	Now using Lemma \ref{pfw_phi}, we get
	\begin{equation*}
		\begin{split}
			h^0(\mb F_e,aC_e+bf_e)&=h^0(\mathbb{P}^1,\pi_\ast(aC_e)\otimes \mathcal{O}_{\mathbb{P}^1}(b))\\
			%	&=h^0(\mathbb{P}^1,(\mathcal{O}_{\mathbb{P}^1}\oplus\mathcal{O}_{\mathbb{P}^1}(-e)\oplus\mathcal{O}_{\mathbb{P}^1}(-2e)\oplus\dots \oplus \mathcal{O}_{\mathbb{P}^1}(-ae)e))\otimes \mathcal{O}_{\mathbb{P}^1}(b))\\
			&=h^0(\mathbb{P}^1,\mathcal{O}_{\mathbb{P}^1}(b)\oplus\mathcal{O}_{\mathbb{P}^1}(b-e)\oplus\mathcal{O}_{\mathbb{P}^1}(b-2e)\oplus\dots \oplus \mathcal{O}_{\mathbb{P}^1}(b-ae))\\
			&=h^0(\mathbb{P}^1,\mathcal{O}_{\mathbb{P}^1}(b))+h^0(\mathbb{P}^1,\mathcal{O}_{\mathbb{P}^1}(b-e))+\dots+h^0(\mathbb{P}^1,\mathcal{O}_{\mathbb{P}^1}(b-ae))\\
			&=(b+1)+(b-e+1)+\dots+(b-ae+1)\\
			&=(a+1)b+a+1-\frac{a(a+1)}{2}e\\
			&=(a+1)(b+1)-\frac{a(a+1)}{2}e\\
			&=(a+1)(b+1-\frac{ae}{2}).\\
		\end{split}		
	\end{equation*}
\end{proof}

\section{Curves on blowups}\label{curves on blowups}
In this section we obtain explicit descriptions of 
$(-1)$-curves  and $(-2)$-curves on the blowup of 
$r$ points on Hirzebruch surface $\mathbb{F}_e$ when 
$r \leqslant e+4$, thereby, generalizing \cite{HJNS24}.

Let $P:=\{p_1,\ldots,p_{r}\}$ be a subset of $\mb F_e$ such that  
$p_i$ are in different fibers of $\pi$.
Let $\pi_P:\mb F_{e,P}\to \mb F_e$ denote the blowup of $\mb F_e$ 
at the points in $P$. Let $E_i$ be the exceptional divisor over the point $p_i$.
Let $K_e$ denote the canonical divisor of $\mb F_e$ and let 
$K_{e,P}$ denote the canonical divisor of $\mb F_{e,P}$. 
Then we have 
$$K_{e,P}=K_e+\sum_{i=1}^{r}E_i=-2C_e-(e+2)f_e+\sum_{i=1}^{r}E_i\,.$$
Let $P_i:=P\setminus \{p_i\}$. Let $\pi_{P_i}:\mb F_{e,P_i}\to \mb F_e$ denote 
the blowup of $\mb F_e$ at the points in $P_i$.

We prove some results  
which will help us write down all $(-1)$-curves.

\begin{lemma}\label{lem1 -1}
	Let $r\leqslant e+2$. Let $P=\{p_1,\dots,p_{r}\} \subset \mathbb{F}_e$ be 
	$r$ points on $\mathbb{F}_e$ in distinct fibers of the map $\pi$. Let 
	$C$ be a smooth rational curve such that $C^2=-1$. 
	If $\pi_P(C)={\rm pt}$ then $C=E_i$, one of the exceptional 
	divisors. If $\pi_P(C)$ is a curve, then write
	$C=aC_e+bf_e-\sum_{i=1}^rm_iE_i \subset  \mathbb{F}_{e,P}$, 
	where $a,b,m_i$ are non-negative integers (using \cite[Chapter 5, Corollary 2.18]{Ha}). Then one  
	of the following holds:
	\begin{enumerate}[label=(\alph*)]
		\item $a=0,b=1$,
		\item $a=1,b=0$,
		\item $a=1,b=e$.
	\end{enumerate}   
\end{lemma}

\begin{proof}
	It is clear that if $\pi_P(C)={\rm pt}$, then $C=E_i$, one of the exceptional 
	divisors. So let us assume that $\pi_P(C)$ is a curve. Thus, we may write 
	$\pi_P(C)=aC_e+bf_e$. From \cite[Chapter 5, Corollary 2.18]{Ha} it follows that 
	either $a=0,b=1$ or $a=1, b=0$ or $a> 0, b\geqslant  ae$. 
	It suffices to consider the 
	case when $a> 0, b\geqslant  ae$. 
	
	Since $C$ is the strict transform of the irreducible 
	curve $\pi_P(C)$, it follows that we may write 
	$C=aC_e+bf_e-\sum_{i=1}^rm_iE_i$, with $m_i\geqslant 0$.
	By the genus formula we have $K_{e,P}\cdot C=-1$.
	This and the condition $C^2=-1$ gives 
	\begin{equation}\label{sum_mi_r (-1)}
		\sum_{i=1}^{r}m_i=2b+2a-ae-1,
	\end{equation}
	and
	\begin{equation}\label{sum_mi^2 _r (-1)}
		\sum_{i=1}^{r}m_i^2=2ab-a^2e+1.
	\end{equation}
	We may assume 
	$m_i>0$ for all $1\leqslant i \leqslant r$, or else we reduce to $r-1$ case. 
	It is easily checked that the assertion is true when $r=0$. So if 
	one of the $m_i=0$, we are done by induction on $r$. 
	
	Let $C_1$ be the image of $C$ in $\mathbb{F}_{e,P_1}$. We see that 
	$C_1$ is an integral curve with $C_1^2=-1+m_1^2\geqslant 0$. Using Proposition \ref{dim}, 
	we see that $h^0(\mb F_e, C_e+ef_e)=e+2 >r-1$, which guarantees the existence of an 
	effective divisor $D=C_e+ef_e-\sum_{i=2}^{r}n_iE_i$ on $\mathbb{F}_{e,P_1}$, where all 
	$n_i\geqslant 1$. So we have,
	$$0\leqslant C_1\cdot D=b-\sum_{i=2}^{r}n_im_i\leqslant b-\sum_{i=2}^{r}m_i\,.$$
	The strict transform of the 
	fiber of the map $\pi:\mb F_e\to \mb P^1$ passing through the point $p_i$
	is an irreducible curve whose class is $f_e-E_i$. This curve is distinct 
	from $C$, 
	and so we get $C \cdot (f_e-E_i)\geqslant 0$. This shows that $m_i\leqslant a$ 
	for all $i$. 
	As $m_1\leqslant a $, we get
	$$\sum_{i=1}^rm_i\leqslant b+a.$$
	Hence from \eqref{sum_mi_r (-1)},
	$$2a+2b-ae-1\leqslant b+a$$
	Thus $b\leqslant a(e-1)+1$. But as $b\ne 0$ from 
	\cite[Chapter 5, Corollary 2.18]{Ha} we have $b\geqslant ae$. 
	It is now easily checked that $a=1,b=e$ is the only possibility.
\end{proof}

\begin{lemma}\label{lem r+3 C^2=-1}
	Let $r=e+3$. Let $P=\{p_1,\dots,p_{r}\} \subset \mathbb{F}_e$ be $r$ 
	points on $\mathbb{F}_e$ in distinct fibers of the map $\pi$. Let 
	$C$ be a smooth rational curve such that $C^2=-1$. 
	If $\pi_P(C)={\rm pt}$ then $C=E_i$, one of the exceptional 
	divisors. If $\pi_P(C)$ is a curve, then write
	$C=aC_e+bf_e-\sum_{i=1}^rm_iE_i \subset  \mathbb{F}_{e,P}$, 
	where $a,b,m_i$ are non-negative integers. Then  one  
	of the following holds:
	\begin{enumerate}
		\item $a=0,b=1$,
		\item $a=1,b=0$,
		\item $a=1,b=e$,
		\item $a=1,b=e+1$.
	\end{enumerate}   
\end{lemma}
\begin{proof}
	As in the proof of the previous Lemma, we may assume that 
	$a>0, b\geqslant ae$. 
	Consider the situation when $r\in \{e+3,e+4\}$.
	Similar to the previous proof we may assume $m_i>0$ 
	for all $1\leqslant i \leqslant r$. Recall equations 
	\eqref{sum_mi_r (-1)} and \eqref{sum_mi^2 _r (-1)}.
	Let $C_1$ be the image of $C$ in $\mathbb{F}_{e,P_1}$. Then $C_1$ is an 
	integral curve with $C_1^2\geqslant 0$. As earlier, using Proposition \ref{dim} we see that 
	$h^0(\mb F_e, C_e+(e+1)f_e)=e+4 > e+3$, which guarantees  
	$D=C_e+(e+1)f_e-\sum_{i=2}^{r}n_iE_i$ is effective divisor on $\mathbb{F}_{e,P_1}$, where all 
	$n_i\geqslant 1$. Hence
	$$0\leqslant C_1\cdot D=b+a-\sum_{i=2}^{r}n_im_i\leqslant b+a-\sum_{i=2}^{r}m_i\,.$$
	Using $m_1\leqslant a $, we get
	$$\sum_{i=1}^{r}m_i\leqslant b+2a.$$
	Hence, from \eqref{sum_mi_r (-1)},
	$$2a+2b-ae-1\leqslant b+2a.$$
	Thus $b\leqslant ae+1$. But as $b\ne 0$ from 
	\cite[Chapter 5, Corollary 2.18]{Ha} we have $b\geqslant ae$.
	Thus we have two possibilities for the value of $b$ when $r\in\{e+3,e+4\}$.
	We will use this fact in the next Lemma also. 
	
	For the remainder of the proof of this Lemma, we assume that $r=e+3$. 
	
	By Cauchy-Schwarz inequality, we have
	$$\left(\sum_{i=1}^{e+3}m_i\right)^2\leqslant (e+3) \sum_{i=1}^{e+3}m_i^2\,,$$
	that is,
	\begin{equation}\label{Cauchy}
		(2a+2b-ae-1)^2\leqslant (e+3)(2ab-a^2e+1) \text{ (using \eqref{sum_mi_r (-1)} and \eqref{sum_mi^2 _r (-1)})}.
	\end{equation}
	First, consider the case when $b=ae$.
	In this case \eqref{Cauchy} reduces to
	$$(a(e+2)-1)^2\leqslant (e+3)(a^2e+1).$$
	We can see that this is not possible when $a\geqslant 3$. 
	If $a=2$, we have $\sum_{i=1}^{e+3}m_i=2e+3$ and 
	$1\leqslant m_i\leqslant 2$ for all $1\leqslant i \leqslant r$. 
	The only possibility is 
	\begin{itemize}
		\item cardinality of the set $\{i\,|\, m_i=2\}=e$,
		\item cardinality of the set $\{i\,|\, m_i=1\}=3$.
	\end{itemize}
	In this case $\sum_{i=1}^{e+3}m_i^2=4e+3$. 
	But \eqref{sum_mi^2 _r (-1)} says that 
	$\sum_{i=1}^rm_i^2=4e+1$. Hence $a=2$ is not possible. Thus  $a=1,b=e$.
	
	Next consider the case when $b=ae+1$.
	Here \eqref{Cauchy} reduces to
	$$(a(e+2)+1)^2\leqslant (e+3)(a^2e+2a+1)\,,$$
	which simplifies to
	$$a^2e+4a^2\leqslant 2a+2+e.$$
	It is easy to see if $a\geqslant 2$, we get $e$ is negative which is a 
	contradiction. Hence $a=1,b=e+1$.
\end{proof}

\begin{lemma}\label{C^2=-1 e+4}
	Let $r=e+4$. Let $P=\{p_1,\dots,p_{r}\} \subset \mathbb{F}_e$ be $r$ 
	points on $\mathbb{F}_e$ in distinct fibers of 
	the map $\pi$. If $e\geqslant 1$, further assume that 
	none of the $p_i$ are on $C_e$. 
	Let $C$ be a smooth rational curve such that $C^2=-1$. 
	If $\pi_P(C)={\rm pt}$ then $C=E_i$, one of the exceptional 
	divisors. If $\pi_P(C)$ is a curve, then write
	$C=aC_e+bf_e-\sum_{i=1}^rm_iE_i \subset  \mathbb{F}_{e,P}$, 
	where $a,b,m_i$ are non-negative integers. Then one  
	of the following holds:
	\begin{enumerate}[label=(\alph*)]
		\item $a=0, b=1$ ,
		\item $a=1, b=0$,
		\item $a=1,b=e$,
		\item $a=1,b=e+1$,
		\item \label{e+4, exceptional type 1}$e\geqslant 1$, $2\leqslant a\leqslant (e+3)/2$, 
		$C=aC_e+aef_e-\sum_{i=1}^{e+4}m_iE_i$, where 
		\begin{itemize}
			\item cardinality of the set $\{i\,|\, m_i=a-1\}=2a+1$,
			\item cardinality of the set $\{i\,|\, m_i=a\}=e-2a+3$.
		\end{itemize}
	\end{enumerate}   
\end{lemma}
\begin{proof}
	Similar to the earlier cases, we may assume that $a>0, b\geqslant ae$. If one 
	of the $m_i=0$ then we may reduce to the $r-1$ case. Thus, we also assume 
	all $m_i\geqslant 1$. 
	As observed in the proof of Lemma \ref{lem r+3 C^2=-1}, 
	we have two possibilities for the value of $b$, namely, $b=ae$ or $b=ae+1$. 
	If $a=1$, then our list contains the both the possibilities, 
	$b=e$ and $b=e+1$. So let us assume that $a\geqslant 2$ for the remainder 
	of this proof. 
	
	By Cauchy-Schwarz inequality we have
	$$\left(\sum_{i=1}^{e+4}m_i\right)^2\leqslant (e+4) \sum_{i=1}^{e+4}m_i^2\,,$$
	that is,
	\begin{equation}\label{Cauchy1}
		(2a+2b-ae-1)^2\leqslant (e+4)(2ab-a^2e+1).
	\end{equation}
	
	First, consider the case when $b=ae$.
	In this case \eqref{Cauchy1} reduces to
	$$(a(e+2)-1)^2\leqslant (e+4)(a^2e+1).$$
	It is easy to check that this yields $e\geqslant 2a-3$.

	If $e=0$, then the only 
	possible value for $a$ is 1, which has already been considered. 
	So let us assume that $e\geqslant 1$. If 
	$e=1$, then $a\in \{1,2\}$. It is easily checked, using \eqref{sum_mi_r (-1)}, 
	that $a=1$ is not possible,
	and if $a=2$ then $m_i=1$ for all $i$. Thus, in this case, the curve $C$ 
	falls in type \ref{e+4, exceptional type 1} with $e=2a-3$.
	
	Let us consider the case when for all $i$ we have $1\leqslant m_i<a$. 
	As $b=ae$, we have the equalities $\sum_{i=1}^{e+4}m_i^2=a^2e+1$
	and $\sum_{i=1}^{e+4}m_i=a(e+2)-1$. 
	We get 
	$$\sum_{i=1}^{e+4}\frac{m_i}{a}\Big(1-\frac{m_i}{a}\Big)=
	\frac{(2a+1)(a-1)}{a^2}\,.$$
	As $1\leqslant m_i<a$, we get 
	$$(e+4)\frac{(a-1)}{a^2}\leqslant \frac{(2a+1)(a-1)}{a^2}\,.$$
	This gives $e+4\leqslant 2a+1$. This forces $e=2a-3$, and we easily 
	check that $m_i=a-1$ for all $i$. Thus, in this case, the curve $C$ 
	falls in type \ref{e+4, exceptional type 1} with $e=2a-3$.
	
	Next consider the case when one of the $m_i$, say, $m_1=a$. 
	Write 
	$$C=aC_e+aef_e-aE_1-\sum_{i=2}^{e+4}m_iE_i\,.$$ 
	By assumption, $p_1\notin C_e$. 
	Consider the following
	diagram of an elementary transformation 
	\begin{equation}\label{elementary-transformation}
		\xymatrix{
			\mb F_{e,P} \ar[r]^{g'}\ar[d]_{g} & \mb F_{e,P_1}\ar[d]^\pi\\
			\mb F_{e-1,P'}\ar[r]_{\pi'} & \mb P^1.
		}
	\end{equation}
	The map $g'$ is the blowup at the point $p_1$. 
	If $F_1$ is the fiber of $\pi$ containing $p_1$, then the strict transform 
	of $F_1$ has class $f_e-E_1$. Denote this strict transform by $\t F_1\subset \mb F_{e,P}$.  
	The morphism $g$ is the blowup at a point $q$ and 
	has $\t F_1$ as the exceptional divisor. As $C\cdot \t F_1=0$,
	it follows that $g(C)^2=C^2=-1$, that is, $g(C)$ is a $(-1)$-curve on 
	$\mb F_{e-1,P'}$, and $g(C)$ does 
	not contain $q$. Moreover, it is easily checked that 
	$$g(C)=aC_{e-1}+a(e-1)f_{e-1}-\sum_{i=2}^{e+4}m_iE_i\,.$$
	The points in $P'$ satisfy the same property, they are in different fibers 
	of $\pi'$ and none of these are on $C_{e-1}$. 
	Moreover, the cardinality of $P'$ is $e+3=(e-1)+4$. Thus, repeating 
	the above process each time one of the $m_i$ equals $a$, we 
	reach a case $e'\geqslant 1$, with all $1\leqslant m_i<a$, or we reach $e'=1$. 
	However, note that when $e'=1$, we saw earlier that the only possibility is 
	$a=2$, and $m_i=1$ for all $i$. Thus, in both cases we have that all the $m_i$ satisfy 
	$m_i=a-1=(e'+3)/2-1=(e'+1)/2$. Thus, in this case, the curve $C$ 
	fall in type \ref{e+4, exceptional type 1}, with $2\leqslant a=(e'+3)/2\leqslant (e+3)/2$.
	
	Next let us consider the case when $b=ae+1$. 
	In this case \eqref{Cauchy1} reduces to
	$$(a(e+2)+1)^2\leqslant (e+4)(a^2e+2a+1).$$
	It is easy to check that this yields $e\geqslant (2a+1)(2a-3)$. 
	Recall that we have assumed that $a\geqslant 2$, and so we have $e\geqslant 5$. 
	
	Let us consider the case when for all $i$ we have $1\leqslant m_i<a$. 
	We have the equalities $\sum_{i=1}^{e+4}m_i^2=a^2e+2a+1$
	and $\sum_{i=1}^{e+4}m_i=a(e+2)+1$. 
	We get 
	$$\sum_{i=1}^{e+4}\frac{m_i}{a}\Big(1-\frac{m_i}{a}\Big)=
	\frac{(2a+1)(a-1)}{a^2}\,.$$
	As $1\leqslant m_i<a$, we get 
	$$(e+4)\frac{(a-1)}{a^2}\leqslant \sum_{i=1}^{e+4}\frac{m_i}{a}\Big(1-\frac{m_i}{a}\Big)
	=\frac{(2a+1)(a-1)}{a^2}\,.$$
	This gives $(2a+1)(2a-3)+4\leqslant e+4\leqslant 2a+1$. This yields 
	$(2a+1)(a-2)+2\leqslant 0$, which is possible only when $a=1$, 
	which is a contradiction.

	Next consider the case when one of the $m_i$, say, $m_1=a$. 
	Write 
	$$C=aC_e+(ae+1)f_e-aE_1-\sum_{i=2}^{e+4}m_iE_i\,.$$ 
	By a similar argument, as the one contained in the paragraph containing 
	\eqref{elementary-transformation}, we see that $g(C)$ is a $(-1)$-curve on 
	$\mb F_{e-1,P'}$ which does 
	not contain $q$. Moreover, it is easily checked that 
	$$g(C)=aC_{e-1}+(a(e-1)+1)f_{e-1}-\sum_{i=2}^{e+4}m_iE_i\,.$$
	The points in $P'$ satisfy the same property, they are in different fibers 
	of $\pi'$. Moreover, the cardinality of $P'$ is $e+3=(e-1)+4$. Thus, repeating 
	the above process each time one of the $m_i$ equals $a$, we reach a case 
	$e'\geqslant 5$ where all the $m_i$ satisfy 
	$1\leqslant m_i<a$, or we reach $e'=4$. As observed above, in both cases we 
	get a contradiction. 
	
	This completes the proof of the Lemma. 
\end{proof}

\begin{remark}\label{existence -1 e+4 last case}
	Consider the numerical classes of curves $C$ in Lemma \ref{C^2=-1 e+4}\ref{e+4, exceptional type 1}.
	Note that $C^2=-1$ and $-K_{e,P}\cdot C=1$.
	Using Riemann-Roch theorem we have, 
	$$h^0(\mathbb{F}_{e,P},C)\geqslant \frac{C^2-K_{e,P}\cdot C}{2}+1\geqslant 1.$$
	Hence, $C$ is an effective divisor. That is, curves in 
	Lemma \ref{C^2=-1 e+4}\ref{e+4, exceptional type 1} exist.
	However, it is not clear if these curves are reduced or irreducible. 
\end{remark}

The next Lemma will help us write down all $(-2)$-curves. 

\begin{lemma}\label{-2 curves}
	Let $r\leqslant e+4$. Let $P=\{p_1,\dots,p_{r}\} \subset \mathbb{F}_e$ be $r$ 
	points on $\mathbb{F}_e$ in distinct fibers of the map $\pi$. Let 
	$C$ be a smooth rational curve such that $C^2=-2$. 
	Then $\pi_P(C)$ is a curve. Write
	$C=aC_e+bf_e-\sum_{i=1}^rm_iE_i \subset  \mathbb{F}_{e,P}$, 
	where $a,b,m_i$ are non-negative integers. Then one  
	of the following holds:
	\begin{enumerate}[label=(\alph*)]
		\item $a=0,b=1$, 
		\item $a=1,b=0$, 
		\item $a=1,b=e$, 
		\item $a=1,b=e+1$.
	\end{enumerate}
\end{lemma}
\begin{proof}
	If $\pi_P(C)$ is a point, then $C$ is forced to be $E_i$ 
	for some $i$. But then $C^2=-1$, which is a contradiction. 
	Since $C^2=-2$ and $C$ is a smooth rational curve, 
	from the genus formula it follows that $K_{e,P} \cdot C=0$.
	
	Suppose $P$ is not a subset of $\pi_P(C)$. Then there is $i$ such that 
	$p_i\notin \pi_P(C)$. It is clear that $m_i=0$ and that $C$ is
	the strict transform of $C_i$, an irreducible smooth rational curve on 
	$\mb F_{e,P_i}$, and which satisfies $C_i^2=-2$, and we are done by induction 
	on $r$. The base case for the induction can be taken to be $r=0$,
	in which case the assertion is clearly true. 
	
	So consider the case when $C'=\pi_P(C)$ is a curve containing 
	the set $P$. Then $C=aC_e+bf_e-\sum_{i=1}^{r}m_iE_i$, where $m_i\geqslant 1$ 
	as $P\subset \pi_P(C)$. 
	Assume that one of the $m_i=1$. 
	Then the image of $C$ in $\mb F_{e,P_i}$ is a smooth rational
	curve $C'$ such that $C'^2=-1$. 
	Moreover, $C$ is the strict transform of the curve $C'$. Thus,
	in this case we are done using Lemma \ref{lem1 -1} and Lemma \ref{lem r+3 C^2=-1}.
	
	Thus, we may further assume that $m_i\geqslant 2$. 
	The conditions $K_{e,P} \cdot C=0$
	and $C^2=-2$ give the following two equations
	\begin{align}\label{e1}
		\sum_{i=1}^{r} m_i&=2b+2a-ae,\\
		\sum_{i=1}^{r} m_i^2&=2ab -a^2e+2.\nonumber
	\end{align}
	
	Similar to the earlier cases, we may assume that 
	$a>0$ and $b\geqslant ae$. Let $C_1$ 
	denote the image of $C$ in $\mb F_{e,P_1}$. Then it is clear that 
	$C_1^2=C^2+m_1^2=-2+m_1^2\geqslant 2$. As $C_1$ is an 
	integral curve such that $C_1^2\geqslant 0$,
	it follows that if $D$ is any effective divisor then $C_1 \cdot D\geqslant 0$. 
	Note that the class of $C_1$ is given by 
	$$aC_e+bf_e-\sum_{i=2}^{r}m_iE_i\,.$$
	Consider the line bundle $D_e=C_e+(e+\lambda)f_e$ on $\mb F_e$, where $\lambda\geqslant 0$
	is an integer. Using Proposition \ref{dim} we see
	that $h^0(\mb F_e, D_e)= e+2+2\lambda$. Taking $\lambda=1$
	we see that there is an effective divisor $D\subset \mb F_e$ passing through 
	the $e+3$ points in the set $P_1$ and linearly equivalent to $D_e$. 
	Thus, if $r\leqslant e+4$, then we can find an effective divisor 
	$D\subset \mb F_e$ passing through 
	the $r-1$ points in the set $P_1$ and linearly equivalent to $D_e$.
	Thus, the strict transform of $D$
	on $\mb F_{e,P_1}$ is of the form 
	$C_e + (e+1)f_e-\sum_{i=2}^{r} n_iE_i$, where all $n_i\geqslant 1$.
	Using $C_1 \cdot D\geqslant0$ we get 
	$$-ae+a(e+1)+b-\sum_{i=2}^{r}n_im_i\geqslant 0\,.$$
	Thus, we get 
	\begin{align*}
		\sum_{i=1}^{r}m_i\leqslant m_1+\sum_{i=2}^{r}n_im_i\leqslant 
		b+a+m_1\leqslant b+2a\,.
	\end{align*}
	Combining this with equation \eqref{e1}, we get 
	$$2b+2a-ae=\sum_{i=1}^{r}m_i\leqslant b+2a\,.$$
	This shows that $b\leqslant ae$. It follows that $b=ae$. 
	Putting this into equation \eqref{e1} we get 
	\begin{equation}\label{e2}
		\sum_{i=1}^{r}m_i=a(e+2)\,,\qquad \qquad \sum_{i=1}^{r}m_i^2=a^2e+2\,.
	\end{equation}
	If $r\leqslant e+3$, then using the Cauchy-Schwarz inequality we get 
	$$a^2(e+2)^2\leqslant (a^2e+2)(r)\leqslant (a^2e+2)(e+3)\,.$$
	From this we get $a^2(e+4)\leqslant 2e+6$. 
	From this we easily conclude that $a=1$, but this is a contradiction 
	as we assumed 
	all $1=a\geqslant m_i\geqslant 2$. This completes the proof of the 
	Lemma when $r\leqslant e+3$. 
	
	It remains to consider the case $r=e+4$. From
	\eqref{e2}, we get 
	$$\sum_{i=1}^{e+4}\frac{m_i}{a}\Big(1-\frac{m_i}{a}\Big)=\frac{2(a^2-1)}{a^2}\,.$$
	First assume that each $1\leqslant m_i<a$. Then we get 
	$$(e+4)\frac{(a-1)}{a^2}\leqslant \frac{2(a^2-1)}{a^2}\,.$$
	This gives $e+4\leqslant 2(a+1)$. Applying Cauchy-Schwarz inequality to 
	\eqref{e2} gives $2a^2\leqslant e+4\leqslant 2(a+1)$. This shows that $a=1$,
	which gives a contradiction as $1=a\geqslant m_i\geqslant 2$. The inequality
	$2a^2\leqslant e+4$ also shows that if $e\leqslant 3$, then $a=1$, which is again a 
	contradiction. Thus, for the rest of the proof we may restrict our attention 
	to the case $e\geqslant 4$.
	
	The same argument as in the paragraph containing \eqref{elementary-transformation}, 
	shows that $g(C)$ is a $(-2)$-curve on $\mb F_{e-1,P'}$ which does 
	not meet the point being blown up by $g$. Moreover, it is easily checked that 
	if $C=aC_e+aef_e-aE_1-\sum_{i=2}^{e+4}m_iE_i$, then 
	$$g(C)=aC_{e-1}+a(e-1)f_{e-1}-\sum_{i=2}^{e+4}m_iE_i\,.$$
	The points in $P'$ satisfy the same property, they are in different fibers 
	of $\pi'$. Moreover, the cardinality of $P'$ is $e+3=(e-1)+4$. Thus, repeating 
	the above process each time one of the $m_i$ equals $a$, we either 
	reach the case $e=3$, or we reach a case where all the $m_i$ satisfy 
	$1\leqslant m_i<a$. In both cases, we get a contradiction, as observed above. 
	This completes the proof of the Lemma. 
\end{proof}

\section{Fixed Components of anticanonical divisors}\label{Fixed Component e+3} 

Consider the line bundle $L=C_e+(e+2)f_e$ on $\mb F_e$. This is a very ample
line bundle with $h^0(\mb F_e, L)=e+6$ (using Proposition \ref{dim}). Two general sections define two smooth curves 
meeting transversally at $L^2$ points, that is, $e+4$ points. Conversely, if we take $r\leqslant e+4$ general 
points, then it is easily checked that:\\\\
$(\dagger)$ the subspace $V\subset H^0(\mb F_e, L)$ 
consisting of sections vanishing at these points is $e+6-r$ dimensional, and
the base locus of $V$ is the set of these $r$ points with the reduced scheme structure.\\\\
Thus, it follows that on $\mb F_{e,P}$, the line bundle $C_e+(e+2)f_e-\sum_{i=1}^rE_i$
is base point free when $r\leqslant e+4$ and $P$ is a set consisting of $r$
points in general position. 

Let $P'=\{p_1,\ldots,p_r,x\}\subset \mb F_e$ be $r+1\leqslant e+4$ points in 
general position. Let $P=\{p_1,\ldots,p_r\}$. We write 
$$-K_{\mb F_{e,P'}}=C_e+C_e+(e+2)f_e-\sum_{i=1}^rE_i-E_x\,.$$
Let $D=D_1+D_2$ where $D_1$ and $D_2$ are effective divisors. Then the 
base locus of $D$ is a subset of the union of the base locus of 
$D_1$ and the base locus of $D_2$. The base locus of $C_e+(e+2)f_e-\sum_{i=1}^rE_i-E_x$
is empty. Thus, the only possible components of the base locus of $-K_{\mb F_{e,P'}}$
are the components of $C_e$. Recall the definition of $\Delta$ from 
Definition \ref{defDelta}. 
If $e>0$, then since none of the points of $P'$
are on $C_e$, it follows that there are no curves in $\Delta$ of the third kind. 
If $e=0$ then $C_e$ is base point free, and so again, there are no curves 
in $\Delta$ of the third kind.

\section{Computation of Seshadri Constants}\label{SCcomp}
In this section we will compute Seshadri constants at a very general point 
on the Hirzebruch surface $\mb F_e$ blown up at $r\leqslant e+3$ points. 
In order to apply the strategy mentioned in Section \ref{preliminaries section}, 
we will need that 
the surface $\mb F_e$ blown up at $r\leqslant e+5$ points has an effective 
anticanonical divisor, for example, see \cite[Proposition 3]{FL2021}. 

Recall the subset $T$ from Proposition \ref{general position}. 
We will say that points $\{p_1,\ldots,p_r\}$ are in very general position
if all the following conditions are satisfied: 
the tuple $(p_1,\ldots,p_r)\in T$, the condition given in 
\S\ref{Fixed Component e+3}$(\dagger)$ holds, the points $p_i$
are in different fibers of the map $\pi$, none of these points are on 
$C_e$ when $e>0$, and when $e=0$ ($\mb F_e=\mb P^1\times \mb P^1$) no two 
of them are in the same vertical or horizontal fibers. It is not difficult 
to see that the set of points satisfying these conditions is the intersection
of $T$ with a nonempty Zariski open subset. 
\begin{theorem}\label{SC e+2}
	For $r=e+2$, let $P= \{p_1,\dots,p_r\}\subset \mathbb{F}_e$ be a set of $r$ distinct points. Let $x \in \mathbb{F}_{e,P}$ be such that $P'=P\cup \{\pi_{P}(x)\}$ be a set of $r+1$ points which are in very general position. Let 
	$L=\alpha C_e+\beta f_e-\sum_{i=1}^r\mu_iE_i$ be an ample line 
	bundle on $\mathbb{F}_{e,P}$. The
	Seshadri constant of $L$ at $x$ is
	$$\varepsilon(\mathbb{F}_{e,P},L,x)=
	\min \left( \alpha,\,\, \beta-\sum_{j=1}^e \mu_{i_j}, \,\,\beta+\alpha-\sum_{i=1}^r\mu_i\right),$$
	where the sum in $\beta-\sum_{j=1}^e \mu_{i_j}$ is taken over the largest $e$ $\mu_i$.
	When $e=0$ this sum is to be taken as $\beta$. 
\end{theorem}
\begin{proof}
	Let $\pi_x: \mathbb{F}_{e,P'} \to \mathbb{F}_{e,P}$ be the 
	blowup of $\mathbb{F}_{e,P}$ at $x$. Let $E_x$ denote the exceptional 
	divisor corresponding to $\pi_x$.
	Using Lemma \ref{lem r+3 C^2=-1}, Lemma \ref{-2 curves} and the discussion in 
	section \ref{Fixed Component e+3},
	we write down members of the set $\Delta$ defined in Definition \ref{defDelta}. From the discussion in 
	section \ref{Fixed Component e+3} it follows that $\Delta$ contains no curves 
	of the third kind. 
	Let $C\subset \mathbb{F}_{e,P}$ be an integral curve through $x$ such 
	that the strict transform $\widetilde{C}$ satsifies:
	$\widetilde{C}$ is a $(-1)$-curve, or a $(-2)$-curve.
	
	Let $\widetilde{C}= aC_e+bf_e-\sum_{i=1}^rm_iE_i-m_xE_x$. Then $m_i\geqslant 0$, $m_x>0$ and $(a,b)\neq (0,0)$. 
	If $\widetilde{C}$ is a $(-1)$-curve then we apply Lemma \ref{lem r+3 C^2=-1}. 
	\begin{itemize}
		\item $(a,b)=(0,1)$. We conclude that $\widetilde{C}$ is the strict transform of a fiber of 
		$\pi$ (recall the map $\pi$ from \eqref{def-pi}). Since $p_i,x$ are in very general position, 
		it easily follows that $\widetilde{C}$ is the strict transform of the fiber of $\pi$ 
		through the point $x$. Thus, $\widetilde{C}=f_e-E_x \in \Delta$. 
		\item $(a,b)=(1,0)$. Here $\widetilde{C}$ is strict transform of $C_e$. If $e\geqslant 1$, since we are assuming 
		that the point $x$ is not in $C_e$, this case is not possible. If $e=0$, then this case is possible
		and $\widetilde{C}=C_e-E_x \in \Delta$.
		\item $(a,b)=(1,e)$. Note that $h^0(\mb F_e, C_e+ef_e)=e+2$. Thus, the curve $C=\pi_P(\widetilde{C})$
		can pass through at most $e+1$ points in $P'$, one of them being $x$. The multiplicity through each 
		point is forced to 1 as $m_i\leqslant a=1$. Thus, $C$ has to pass through exactly $e$ points in $P$ 
		and through $x$, for $\widetilde{C}^2$ to be $-1$. Thus, the possibilities are 
		$\widetilde{C}=C_e+ef_e-\sum_{j=1}^eE_{i_j}-E_x \in \Delta$.
		\item $(a,b)=(1,e+1)$. Arguing as in the previous case, we see that 
		$\widetilde{C}=C_e+(e+1)f_e-\sum_{i=1}^rE_{i}-E_x \in \Delta$.
	\end{itemize}
	If $\widetilde{C}$ is a $(-2)$-curve then we apply Lemma \ref{-2 curves}.
	\begin{itemize}
		\item $(a,b)=(0,1)$. As before, $\widetilde{C}$ is forced to be strict 
		transform of a fiber of $\pi$.
		But this case does not occur as $\widetilde{C}^2=-1$. 
		\item $(a,b)=(1,0)$. Again, $\widetilde{C}$ is the strict transform 
		of $C_e$. If $e\geqslant 1$ then this 
		case is not possible as we have assumed that $x\notin C_e$. If $e=0$, 
		then there is a unique $C_e$ 
		through the point $x$, however, this does not contain any other $p_i$. 
		Thus, $\widetilde{C}^2=-2$ is 
		not possible. 
		\item $(a,b)=(1,e)$. Since $h^0(\mb F_e, C_e+ef_e)=e+2$ and points in $P'$ are in very general 
		position, any curve $D\in |C_e+ef_e|$ can pass through at most $e+1$ points of $P'$. 
		But to have $\widetilde{C}^2=-2$, this has to pass through $e+2$ points, 
		which is not possible  by Proposition \ref{general position}.
		\item $(a,b)=(1,e+1)$. Arguing as in the previous case we see that $\widetilde{C}^2=-2$
		is not possible in this case. 
	\end{itemize}
	From the above discussion, we conclude that for $e>0$
	$$\Delta=\{f_e-E_x,\,\,C_e+ef_e-\sum_{j=1}^eE_{i_j}-E_x,\,\,C_e+(e+1)f_e-\sum_{i=1}^rE_{i}-E_x\}\,.$$
	If $e=0$, then 
	$$\Delta=\{f_e-E_x,\,\,C_e-E_x,\,\,C_e+f_e-\sum_{i=1}^rE_{i}-E_x\}\,.$$
	From Proposition \ref{SC computation-1} it follows that when $e>0$ we have 
	$$\varepsilon(\mathbb{F}_{e,P},L,x)=\min \left( \alpha,\,\, \beta-\sum_{j=1}^e \mu_{i_j},\,\, \beta+\alpha-\sum_{i=1}^r\mu_i\right)\,,$$
	where the sum in $\beta-\sum_{j=1}^e \mu_{i_j}$ is over the largest $e$ $\mu_i$. When $e=0$ we have 
	$$\varepsilon(\mathbb{F}_{e,P},L,x)=\min \left( \alpha,\,\, \beta-\alpha e,\,\,\beta-\sum_{j=1}^e \mu_{i_j}, \,\, \beta+\alpha-\sum_{i=1}^r\mu_{i}\right)=
	\min \left( \alpha, \,\, \beta,\,\,\beta+\alpha-\sum_{i=1}^r\mu_i\right)\,.$$
	The second equality holds when $e=0$, as $\beta-\alpha e=\beta$ and $\beta-\sum_{j=1}^e \mu_{i_j}=\beta$ as 
	the sum is over the largest $e$ $\mu_i$. 
	This completes the proof of the Theorem. 
\end{proof}
For $r=e+3$, let $P= \{p_1,\dots,p_r\}\subset \mathbb{F}_e$ be a set of $r$ distinct points. Let $x \in \mathbb{F}_{e,P}$ be such that $P'=P\cup \{\pi_{P}(x)\}$ be a set of $r+1$ points which are in very general position.
Let 
$L=\alpha C_e+\beta f_e-\sum_{i=1}^r\mu_iE_i$ be an ample line 
bundle on $\mathbb{F}_{e,P}$. 
Define
\begin{align}\label{def A(L)}
	A(L)&=\min_{2\leqslant  a\leqslant \frac{e+3}{2}}\left\{\frac{ a\beta-\sum_{j=1}^{2a}(a-1)\mu_{i_j}-\sum_{k=1}^{e-2a+3}a\mu_{i_k}}{a-1}\middle| \{i_j\}\cup\{i_k\}=\{1,2,\dots,r\}\right\}\,,\\
	B(L)&=\min_{2\leqslant a\leqslant\frac{e+2}{2}}\left\{\frac{ a\beta-\sum_{j=1}^{2a+1}(a-1)\mu_{i_j}-\sum_{k=1}^{e-2a+2}a\mu_{i_k}}{a}\middle| \{i_j\}\cup\{i_k\}=\{1,2,\dots,r\}\right\}\,.\nonumber
\end{align}

\begin{theorem}\label{SC e+3}
	With notation as above, the
	Seshadri constant of $L$ at $x$ is
	\begin{align*}
		\varepsilon(\mathbb{F}_{e,P},L,x)&=\min \left( \alpha,\,\, \beta,\,\, \beta+\alpha-\sum_{j=1}^{2}\mu_{i_j}\right)\,,\qquad \qquad \text{for $e=0$}\,,\\
		\varepsilon(\mathbb{F}_{e,P},L,x)&=\min \left( \alpha,\,\,\beta-\sum_{j=1}^e \mu_{i_j}, \,\, \beta+\alpha-\sum_{i=1}^{e+2}\mu_{i_j},\,\, A(L), \,\,B(L)\right)\,,\qquad \qquad \text{for $e>0$}\,,
	\end{align*}
	where the sum in $\beta+\alpha-\sum_{j=1}^{2}\mu_{i_j}$ is over the two largest $\mu_i$, 
	the sum in $\beta+\alpha-\sum_{j=1}^{e}\mu_{i_j}$ is over the largest $e$ $\mu_i$,
	and the sum in $\beta+\alpha-\sum_{j=1}^{e+2}\mu_{i_j}$ is over the largest $e+2$ $\mu_i$.
\end{theorem}
\begin{proof}
	As the proof is very similar to that of Theorem \ref{SC e+2}, we will only sketch
	the main points, leaving the details to the reader. 
	Let $\widetilde{C}= aC_e+bf_e-\sum_{i=1}^rm_iE_i-m_xE_x$. Then 
	$m_i\geqslant 0$, $m_x>0$ and $(a,b)\neq (0,0)$. 
	If $\widetilde{C}$ is a $(-1)$-curve then we apply Lemma \ref{C^2=-1 e+4}. 
	\begin{itemize}
		\item $(a,b)=(0,1)$. Same reasoning as before shows $\widetilde{C}=f_e-E_x \in \Delta$. 
		\item $(a,b)=(1,0)$. Here $\widetilde{C}$ is strict transform of $C_e$. If 
		$e\geqslant 1$, since we are assuming that the point $x$ is not in $C_e$, this case is 
		not possible. If $e=0$, then this case is possible
		and $\widetilde{C}=C_e-E_x \in \Delta$.
		\item $(a,b)=(1,e)$. Same reasoning as before shows
		$\widetilde{C}=C_e+ef_e-\sum_{j=1}^eE_{i_j}-E_x \in \Delta$.
		\item $(a,b)=(1,e+1)$. The curve $C$ can pass through only $e+3$ 
		points, of which one has to be $x$. We see that 
		$\widetilde{C}=C_e+(e+1)f_e-\sum_{j=1}^{e+2}E_{i_j}-E_x \in \Delta$.
		\item Finally we need to consider possible curves in 
		Lemma \ref{C^2=-1 e+4}\ref{e+4, exceptional type 1}.
		In this case we have $e\geqslant 1$, $2\leqslant a\leqslant (e+3)/2$. 
		By Remark \ref{existence -1 e+4 last case}, we know that there is an effective divisor 
		$C\subset \mb F_{e,P}$, passing through $x$, whose multiplicity at 
		$x$ is either $a$ or is $a-1$. 
		The multiplicity can be $a$ only when $e>2a-3$. 
		However, it is not clear if this curve is irreducible or reduced. 
		Thus, for the case $r=e+3$
		we will use a set $\Lambda$ and Proposition \ref{SC Computation}. 
		If $e=2a-3$ the only possible class in $\Lambda$
		coming from this case is: 
		$$aC_e+aef_e-\sum_{j=1}^{2a}(a-1)E_{i_j}-(a-1)E_x\,.$$
		If $e>2a-3$, then $\Lambda$ will contain the following two classes: 
		\begin{itemize}
			\item[$\blacktriangle$] $aC_e+aef_e-\sum_{j=1}^{2a}(a-1)E_{i_j}-\sum_{k=1}^{e-2a+3}aE_{i_k}-(a-1)E_x$
			\item[$\blacktriangle$] $aC_e+aef_e-\sum_{j=1}^{2a+1}(a-1)E_{i_j}-\sum_{k=1}^{e-2a+2}aE_{i_k}-aE_x$
		\end{itemize}
	\end{itemize}
	If $\widetilde{C}$ is a $(-2)$-curve then we apply Lemma \ref{-2 curves}.
	\begin{itemize}
		\item $(a,b)=(0,1)$. As before, $\widetilde{C}$ is forced to 
		be strict transform of a fiber of $\pi$. 
		But this case does not occur as $\widetilde{C}^2=-1$. 
		\item $(a,b)=(1,0)$. Again, $\widetilde{C}$ is the strict transform 
		of $C_e$. If $e\geqslant 1$ then this case is not possible as we have 
		assumed that $x\notin C_e$. If $e=0$, then there is a unique $C_e$ 
		through the point $x$, however, this does not contain any 
		other $p_i$. Thus, $\widetilde{C}^2=-2$ is not possible. 
		\item $(a,b)=(1,e)$. Since $h^0(\mb F_e,C_e+ef_e)=e+2$ and points in 
		$P'$ are in very general position, any curve $C\in |C_e+ef_e|$ can pass 
		through at most $e+1$ points of $P'$. But then 
		$\widetilde{C}^2\geqslant -1$. 
		\item $(a,b)=(1,e+1)$. Since $h^0(\mb F_e,C_e+(e+1)f_e)=e+4$, the 
		curve $C$ can pass through at most $e+3$ points in $P'$. However, 
		in this case we get $\widetilde{C}^2\geqslant -1$. 
	\end{itemize}
	From the above discussion we conclude that for $e=0$ 
	$$\Lambda=\{f_e-E_x,\,\,C_e-E_x,\,\,C_e+(e+1)f_e-\sum_{j=1}^{e+2}E_{i_j}-E_x \}\,.$$
	If $e\geqslant 1$ then 
	\begin{align*}
		\Lambda=\Bigg\{f_e-E_x,\,\,&C_e+ef_e-\sum_{j=1}^eE_{i_j}-E_x,\,\,
		C_e+(e+1)f_e-\sum_{j=1}^{e+2}E_{i_j}-E_x\Bigg\}\bigcup\\
		& \Bigg\{aC_e+aef_e-\sum_{j=1}^{2a}(a-1)E_{i_j}-\sum_{k=1}^{e-2a+3}aE_{i_k}-(a-1)E_x\,\,\vert\,\,2\leqslant a\leqslant(e+3)/2\Bigg\}\bigcup\\
		&\Bigg\{aC_e+aef_e-\sum_{j=1}^{2a+1}(a-1)E_{i_j}-\sum_{k=1}^{e-2a+2}aE_{i_k}-aE_x\,\,\vert\,\,2\leqslant a\leqslant(e+3)/2\Bigg\}\,.
	\end{align*}
	Now we apply Proposition \ref{SC Computation}. When $e=0$ we have 
	$$\varepsilon(\mathbb{F}_{e,P},L,x)=\min \left( \alpha,\,\, \beta,\,\, \beta+\alpha-\sum_{j=1}^{2}\mu_{i_j}\right)\,,$$
	where the sum in $\beta+\alpha-\sum_{j=1}^{2}\mu_{i_j}$ is over 
	the two largest $\mu_i$. When $e>0$ we have 
	$$\varepsilon(\mathbb{F}_{e,P},L,x)=\min \left( \alpha,\,\,\beta-\sum_{j=1}^e \mu_{i_j}, \,\, \beta+\alpha-\sum_{i=1}^{e+2}\mu_i,\,\, A(L),\,\, B(L)\right)\,,$$
	where the sum in $\beta+\alpha-\sum_{j=1}^{e}\mu_{i_j}$ 
	is over the largest $e$ $\mu_i$ 
	and the sum in $\beta+\alpha-\sum_{j=1}^{e+2}\mu_{i_j}$ 
	is over the largest $e+2$ $\mu_i$. 
	This completes the proof of the Theorem. 
\end{proof}

\begin{remark}(\cite[Question 3.12]{HJNS24}
	The Seshadri constants of ample line bundles on $\mathbb{F}_{e,P}$ at very general points are always integers when 
	$r \leqslant e + 2$, where $P= \{p_1,\dots,p_r\}\subset \mathbb{F}_e$ 
	be a set of very general points. When $r=e+3$, the Seshadri constant 
	can possibly be non-integer but still are rational numbers.  
	
\end{remark}

\begin{theorem}\label{Conj 4.8 HJNS}
	Let $P=\{p_1,\dots,p_r\}\subset \mb F_e$ be $r \leqslant e+4$ 
	points in very general  position. Recall that $\pi_P: \mathbb{F}_{e,P} \to \mathbb{F}_e$ 
	is the blowup of $\mathbb{F}_e$ at the points in $P$. Let $C$ be an integral curve on 
	$\mathbb{F}_{e,P}$ such that $C^2<0$. Then $C$ is either the strict transform of 
	$C_e$ or a $(-1)$-curve.
\end{theorem}
\begin{proof}
	Let $C=aC_e+bf_e-\sum_{i=1}^rm_iE_i$ be an integral curve on $\mathbb{F}_{e,P}$ 
	with $C^2<0$. If $\pi_P(C)$ is a point, then $C$ is forced to be one of the exceptional 
	divisors $E_i$ and so the assertion is true. Let us assume that $\pi_P(C)$ is a curve.
	Then $C$ is the strict transform of $\pi_P(C)$, and so we get all the $m_i\geqslant 0$. 
	
	Let $C^2=-k$ and $K_{e,P}\cdot C=j$. Using  $K_{e,P}\cdot C=j$ we get
	\begin{equation}\label{eqn(-)}
		2a+2b-ae+j=\sum_{i=1}^rm_i.
	\end{equation}
	Suppose $a>1$ or $b>e+2$. 
	Note that $h^0(\mb F_e, C_e+(e+2)f_e)=e+6$, so there exists an effective 
	divisor $D\in |C_e+(e+2)f_e|$ such that $D$ passes through all 
	points in $P$. As $\pi_P(C)$ is integral,  $D$ and $\pi_P(C)$ can have a common component 
	only if $\pi_P(C)$ is a component of $D$. But note that $\pi_P(C)\equiv aC_e+bf_e$ and hence $D-
	\pi_P(C)\equiv (1-a)C_e+(e+2-b)f_e$. Since $a>1$ or $b>e+2$, we can see that $D-\pi_P(C)$ 
	cannot be effective as either $1-a$ or $e+2-b$ is negative. So $D$ and $\pi_P(C)$ do not 
	have a common component.  Hence we have 
	$$D\cdot \pi_P(C)\geqslant \sum_{p\in P}(\text{mult}_pD)(\text{mult}_p\pi_P(C)).$$
	Since $\text{mult}_pD\geqslant 1$, we have, 
	$$b+2a\geqslant \sum_{i=1}^rm_i\,.$$
	Combining this and \eqref{eqn(-)} we have 
	$$2a+2b-ae+j=\sum_{i=1}^rm_i \leqslant b+2a\,,$$
	that is, $b-ae+j\leqslant 0$. Since $b-ae\geqslant 0$, 
	we have $j\leqslant0$.
	
	Now by the genus formula we have $$0>j-k=K_{e,P}\cdot C+C^2\geqslant -2$$
	This forces $C$ to be a smooth rational curve and $j-k=-2$.  So $C^2$ is $-1$ or $-2$. 
	From Lemma \ref{-2 curves} we can see that $C^2=-2$ cannot occur in this case, 
	since all curves listed in Lemma \ref{-2 curves} have $a\leqslant 1$ and 
	$b \leqslant e+1$.
	
	Next consider the case $a=1,b=e+2$. Then $C^2=e+4-\sum_{i=1}^rm_i^2$ and $m_i\leqslant 1$ for 
	all $1\leqslant i \leqslant r$. So $C^2\geqslant 0$ in this case.
	
	Next consider the case $a=1,b=e+1$.	Since $h^0(\mb F_e, C_e+(e+1)f_e)=e+4$ and points 
	in $P$ are in very general  position, by Propositon \ref{general position} at most $e+3$ points can lie on $\pi_P(C)$. Since $C^2<0$, 
	and again $0\leqslant m_i\leqslant 1$, we get that $\pi_P(C)$ passes through exactly 
	$e+3$ points of $P$ and hence $C^2=-1$. As $a=1$ and $\pi_P(C)$ is integral, it is forced 
	to be smooth and rational. It easily follows that $C$ is a $(-1)$-curve.
	
	Next consider the case $a=1,b=e$. 
	With similar arguments in above case and using the fact $h^0(\mb F_e, C_e+ef_e)=e+2$, we conclude 
	$\pi_P(C)$ passes through exactly $e+1$ points of $P$ and thus $C$ is a $(-1)$-curve.
	
	If $a=1,b=0$ then $C$ is strict transform of $C_e$. If $a=0,b=1$, then 
	$C$ is the strict transform of a fibre of $\pi$.
	Since no two points in $P$ lies in same fibre we have $C$ is a $(-1)$-curve.
	This completes the proof of the Theorem. 
\end{proof}

\section{Ruled Surfaces}
Let $\Gamma$ be an irreducible smooth projective curve of genus $g\geqslant 1$. 
Let $\mc L$ be a line bundle on $\Gamma$ with degree ${\rm deg}(\mc L)=:-e<0$. Let 
$$\pi: X=\mathbb{P}(\mathcal{O}_{\Gamma}\oplus \mathcal{L}) \to \Gamma$$ 
be the projective bundle, which is a ruled surface. The quotient 
$\mathcal{O}_{\Gamma}\oplus \mathcal{L}\to \mc L$ gives rise to a section 
of the map $\pi$, the image of which we denote by $C_e$. 
Let $K_\Gamma$ denote the 
canonical bundle of $\Gamma$. 
Note that the canonical divisor of $X$ is $K_X=-2C_e+\pi^*(K_{\Gamma}+\mathcal{L})$. 

\begin{comment}
	\begin{proposition}\label{anticanonicalRS}
		With notation as above, we have
		$h^0(X,-K_X)\geqslant e+3(1-g)\,.$
	\end{proposition}
	\begin{proof}
		Using \cite[Chapter 5, Proposition 2.6]{Ha} we see that 
		$\mc O_X(C_e)\cong \mc O_X(1)$. From the projection formula 
		and \cite[Chapter 5, Lemma 2.4]{Ha} we have 
		$$H^0(X,-K_X)=H^0(X,\mathcal{O}_X(2)\otimes\pi^*(K_{\Gamma}^{\vee}\otimes \mathcal{L}^{\vee}))=
		H^0(\Gamma,\pi_\ast(\mathcal{O}_X(2))\otimes K_{\Gamma}^{\vee}\otimes \mathcal{L}^{\vee}).$$
		Now by \cite[Chapter 2, Proposition 7.11]{Ha} we know $\pi_{\ast}(\mathcal{O}_X(2))={\rm Sym}^2(\mathcal{O}_{\Gamma}\oplus \mathcal{L})$. 
		$${\rm Sym}^2(\mathcal{O}_{\Gamma}\oplus \mathcal{L})=\mathcal{O}_{\Gamma}\oplus \mathcal{L}\oplus \mathcal{L}^{\otimes2}\,.$$
		Hence,
		\begin{align*}
			h^0(X,-K_X)&=h^0(\Gamma,K_{\Gamma}^{\vee}\otimes \mathcal{L}^{\vee})\oplus h^0(\Gamma,K^{\vee}_{\Gamma})\oplus h^0(\Gamma,K^{\vee}_{\Gamma}\otimes \mathcal{L})\\
			&\geqslant h^0(\Gamma,K_{\Gamma}^{\vee}\otimes \mathcal{L}^{\vee})\\
			&\geqslant \text{deg }(K_{\Gamma}^{\vee}\otimes \mathcal{L}^{\vee})+1-g 
			\qquad \qquad\text{ (By Riemann-Roch theorem)}\\
			&= 2-2g+e+1-g=e+3(1-g).
		\end{align*}
		This completes the proof of the Proposition. 
	\end{proof}
	\begin{corollary}
		If $e+3(1-g)>0$ then $-K_X$ is an effective divisor. 
	\end{corollary}
	\begin{remark}
		The above proof also shows that $h^0(X,\mc O_X(C_e))=1$.
	\end{remark}
\end{comment}

\begin{lemma} \label{RS rational curves}
	Let $\pi_P: X_P\to X$ be the blowup of $X$ at the points in 
	$P=\{p_1,\dots,p_r\}\subset X$. Let $E_i$ be the exceptional divisor 
	corresponding to $p_i$. If $C$ is a smooth rational curve on $X_P$ then 
	$\pi_P(C)$ is either a fiber in $X$ or it equals one of the points $p_i$. 
	In particular, smooth rational curves 
	on $X_P$ are either $E_i$ or strict transforms of fibers of $\pi$.
\end{lemma}

\begin{proof}
	Follows easily using the fact that we cannot have a nonconstant map $C\to \Gamma$. 
\end{proof}

\begin{lemma}\label{RS fixed component}
	Let $0\leqslant r<e+2-3g$. Let $\pi_P: X_P\to X$ be 
	the blowup of $X$ at the points in 
	$P=\{p_1,\dots,p_r\}\subset X$. Let $E_i$ be the exceptional divisor
	corresponding to $p_i$. For any $x\in X_P$, let 
	$\pi_x: X_{P,x} \to X_P$ be the blowup of $X_P$ at $x$ 
	with exceptional divisor $E_x$. Then $-K_{X_{P,x}}$ is effective. 
	If $C$ is a fixed component of 
	$|-K_{X_{P,x}}|$, then $C$ is one of the following:
	\begin{enumerate}
		\item strict transform of $C_e\subset X$ in $X_{P,x}$,
		\item strict transform of $F_i\subset X$ (the fiber of $\pi$ 
		passing through $p_i$) in $X_{P,x}$,
		\item strict transform of $F_x\subset X$ (the fiber of $\pi$ 
		passing through $\pi_P(x)$) in $X_{P,x}$, 
		\item strict transform of $E_i\subset X_P$ (the exceptional 
		divisor over $p_i$) in $X_{P,x}$, 
		\item $E_x$,
		\item curves with numerical class $f_e$.
	\end{enumerate}
\end{lemma}
\begin{proof}
	Recall the maps $X_P\xrightarrow{\pi_P} X\xrightarrow{\pi}\Gamma$. 
	Let $Q$ denote the effective divisor $\sum_{i=1}^r\pi(p_i)+(\pi\circ\pi_P)(x)$ 
	on $\Gamma$. Note that ${\rm deg}(Q)= r+1$. 
	Let $F_i$ denote the fiber of $\pi$ through the point $p_i$ 
	and let $F_x$ denote the fiber of $\pi$ through the point $\pi_P(x)$. 
	Then $\pi^*Q=\sum_{i=1}^rF_i+F_x$. 
	By Riemann-Roch theorem we have 
	$$h^0(-K_{\Gamma}-\mathcal{L}-Q)\geqslant 
	{\rm deg}(-K_{\Gamma}-\mathcal{L}-Q)+1-g= (2(1-g)+e-r-1)+1-g>0\,,$$
	by the assumption on $r$. Hence, there exists an effective divisor $D'\in |-K_{\Gamma}-\mathcal{L}-Q|$.
	Let us write $D'=D''+D'''$, where $D''$ has support away from the set 
	$\{\pi(p_1),\ldots,\pi(p_r),(\pi\circ\pi_P)(x)\}$ and $D'''$ has support 
	in this set. 
	Thus, we may write 
	\begin{align*}
		-K_X&=2C_e+\pi^*(-K_\Gamma-\mc L-Q)+\pi^*Q\\
		&=2C_e+\pi^*D'+\sum_{i=1}^rF_i+F_x\,\\
		&=2C_e+\pi^*D''+\pi^*D'''+\sum_{i=1}^rF_i+F_x\,.
	\end{align*}
	From this we get that 
	$$-K_{X_P}=-\pi_P^*K_X-\sum_{i=1}^rE_i=2\pi_P^*C_e+\pi_P^*\pi^*D''+\pi_P^*\pi^*D'''+\sum_{i=1}^r\widetilde F_i+\pi_P^*F_x\,.$$ 
	Finally, we conclude that
	\begin{align*}
		-K_{X_{P,x}}&=2(\pi_P\circ\pi_x)^*C_e+(\pi_P\circ\pi_x)^*\pi^*D''+(\pi_P\circ\pi_x)^*\pi^*D'''+\pi_x^*\sum_{i=1}^r\widetilde F_i+((\pi_P\circ\pi_x)^*F_x-E_x)\,.
	\end{align*}
	It is easy to see that $(\pi_P\circ\pi_x)^*F_x-E_x$ is an effective divisor. 
	This proves that $-K_{X_{P,x}}$ is an effective divisor. 
	
	If $f:\widetilde Y \to Y$ is the blowup of a point on a smooth surface, and $D\subset Y$ 
	is an integral curve, then the components of $f^*D$ are the strict transform of $D$
	and possibly the exceptional divisor. Thus, for example, the possible components of 
	$\pi_P^*C_e$ are the strict transform $\widetilde C_e\subset X_P$ and the $E_i$. The possible components 
	of $(\pi_x\circ\pi_P)^*C_e$ are the strict transform $\widetilde C_e\subset X_{P,x}$,
	the strict transforms $\widetilde E_i$ and $E_x$. Arguing similarly, the Lemma follows
	easily. 
\end{proof}

\begin{theorem}\label{SC for ruled sur}
	Let $0\leqslant r < e+2-3g$. Let $P=\{p_1,\dots,p_r\}$ be a set of $r$ 
	points in $X$ such that no two $p_i$ are on the same fiber and 
	$p_i\notin C_e$. Let $\pi_P: X_P\to X$ be the blowup of $X$ at points 
	in $P$.
	For an ample line bundle $L \equiv \alpha C_e+\beta f_e-\sum_{i=1}^r\mu_iE_i$ and for $x\in X_P$ we have the following:
	\begin{enumerate}
		\item If $x$ is not in any of the curves $\widetilde{C_e}, \widetilde{F_1},\dots,\widetilde{F_r},E_1,\dots,E_r$, then $\varepsilon(X_P,L,x)=\alpha$.
		\item If $x$ is not in any of the curves $\widetilde{C_e},E_1,\dots,E_r$, but $x\in \widetilde{F_i}$ for some $i$, then $\varepsilon(X_P,L,x)=\alpha-\mu_i$.
		\item If $x$ is not in any of the curves $\widetilde{C_e},\widetilde{F_1},\dots,\widetilde{F_r}$, but $x\in E_i$ for some $i$, then $\varepsilon(X_P,L,x)=\mu_i$.
		\item If $x$ is not in any of the curves $\widetilde{F_1},\dots,\widetilde{F_r},E_1,\dots,E_r$, but $x\in \widetilde{C_e}$, then $\varepsilon(X_P,L,x)=\min(\beta -\alpha e,\alpha)$.
		\item If $x\in C_e\cap \widetilde{F_i}$, then $\varepsilon(X_P,L,x)=\min(\beta -\alpha e,\alpha-\mu_i)$.
		\item If $x\in \widetilde{F_i}\cap E_i$, then
		$\varepsilon(X_P,L,x)=\min(\alpha-\mu_i,\mu_i)$.
	\end{enumerate}
	Here $F_i$ is the fibre of $\pi$ passing through $p_i$ and $E_i$ is the exceptional divisor on $X_P$ corresponding to $p_i$. 
\end{theorem}
\begin{proof}
	The proof is very similar to the proof of 
	Theorem \ref{SC e+2} and Theorem \ref{SC e+3}.
	As we did there, we write down the 
	set $\Delta$ (see Deefinition \ref{defDelta}) and apply Proposition \ref{SC computation-1}.
	By Lemma \ref{RS fixed component}, $-K_{X_{P,x}}$ is effective. 
	Let $\pi_x:X_{P,x}\to X_P$ denote the blowup of $X_P$ 
	at $x$ with exceptional divisor $E_x$. 
	We need to list integral curves $C\subset X_P$, passing through 
	$x$, such that its strict transform 
	$\widetilde{C}$ is a $(-1)$-curve or a $(-2)$-curve 
	or a fixed component of $|-K_{X_{P,x}}|$. \\
	\textbf{Case 1:} $x$ is not in any of the curves in 
	the following set $\{\widetilde{C_e}, \widetilde{F_1},\dots,\widetilde{F_r},E_1,\dots,E_r\}$. 
	It is easily checked using Lemma \ref{RS rational curves} 
	and Lemma \ref{RS fixed component} that 
	$\Delta$ consists of only one curve, namely, the strict transform of $F_x\subset X$
	in $X_{P,x}$. Its numerical class equals $f_e-E_x$. 
	Hence, by Proposition \ref{SC computation-1} 
	$$\varepsilon(X_P,L,x)=\alpha\,.$$
	\textbf{Case 2:} $x$ is not in any of the curves $\{\widetilde{C_e},E_1,\dots,E_r\}$, 
	but $x\in \widetilde{F_i}$ for some $i$.
	Here $\Delta$ consists of $f_e-E_i-E_x$. Hence, 
	$$\varepsilon(X_P,L,x)=\alpha-\mu_i\,.$$
	\textbf{Case 3:}  $x$ is not in any of the curves 
	$\{\widetilde{C_e},\widetilde{F_1},\dots,\widetilde{F_r}\}$, 
	but $x\in E_i$ for some $i$.
	Here $\Delta$ consists of $E_i-E_x$. 
	Hence,
	$$\varepsilon(X_P,L,x)=\mu_i\,.$$
	\textbf{Case 4:} $x$ is not in any of the curves $\{\widetilde{F_1},\dots,\widetilde{F_r},E_1,\dots,E_r\}$, but $x\in \widetilde{C_e}$.
	Here $\Delta$ consists of $F_x-E_x$ and $\widetilde{C_e}-E_x$. Hence,
	$$\varepsilon(X_P,L,x)=\min(\beta -\alpha e,\alpha)\,.$$
	\textbf{Case 5:} $x\in \widetilde{C_e} \cap \widetilde{F_i}$.
	Here $\Delta$ consists of $\widetilde{C_e}-E_x$,
	$\widetilde{F_i}-E_i-E_x$. Hence,
	$$\varepsilon(X_P,L,x)=\min(\beta -\alpha e,\alpha-\mu_i)\,.$$
	\textbf{Case 6:} $x\in \widetilde{F_i}\cap E_i$.
	Here $\Delta$ consists of $\widetilde{F_i}-E_i-E_x$, $E_i-E_x$. Hence,
	$$\varepsilon(X_P,L,x)=\min(\alpha-\mu_i,\mu_i)\,.$$
\end{proof}

\begin{proposition}\label{Conj 2.1 HJNS}
	Let $\phi:X=\mb P(E)\to \Gamma$ be a ruled surface over a smooth curve $\Gamma$ of genus $g$,
	with invariant $e>0$.
	Let $P=\{p_1,\dots,p_r\}\subset X$ be $r \leqslant e$ distinct 
	points in $X$. Let $\pi_P: X_P \to X$ 
	be the blowup of $X$ at the points in $P$. Let $C$ be a integral curve on 
	$X_P$ such that $C^2<0$. Then $C$ is one of the following:
	\begin{enumerate}
		\item $\widetilde{C}_e$,
		\item $E_i$, where $E_i$ is the exceptional divisor corresponding 
		to $p_i\in P$ for some $1\leqslant i \leqslant r$,
		\item strict transform of the fiber containing  
		$p_i\in P$ for some $1\leqslant i \leqslant r$.
	\end{enumerate}
\end{proposition}

\begin{proof}
	%	From the assumption $r\leqslant e+5-6g$ and 
	%   Proposition \ref{anticanonicalRS}, we can see 
	%   that surface $X_P$ is anticanonical. 
	
	Let $C=aC_e+bf_e-\sum_{i=1}^rm_iE_i$ be an integral curve on $X_P$ 
	with $C^2<0$. If $\pi_P(C)$ is a point, then $C$ is forced to be one of the exceptional 
	divisors $E_i$ and so the assertion is true. Let us assume that $\pi_P(C)$ is a curve.
	Then $C$ is the strict transform of $\pi_P(C)$, and so we get all the $m_i\geqslant 0$. 
	Clearly if $\pi_P(C)$ is $C_e$ or a fiber we are done. So
	suppose that $\pi_P(C)$ is different from $C_e$ and fiber. 
	The integral curve $\pi_P(C)\subset X$ has numerical class $aC_e+bf_e$, 
	where $a>0$ and $b\geqslant ae$, see \cite[Chapter 5, Proposition 2.20(a)]{Ha}. 
	The self-intersection number of $\pi_P(C)$ is 
	$$(aC_e+bf_e)^2=ab+a(b-ae)\geqslant ab\geqslant a^2e\,.$$ 
	Now since $m_i\leqslant a$ (obtained by intersecting $C$ with the strict 
	tranform of a fiber), we get that 
	$$C^2=(aC_e+bf_e)^2-\sum_{i=1}^r m_i^2\geqslant a^2e-\sum_{i=1}^rm_i^2\geqslant 0\,.$$
	The last inequality holds since $r\leqslant e$. This is a contradiction 
	to the assumption that $C^2<0$. This completes the proof of the Proposition. 
\end{proof}

\begin{corollary}\label{cor conj 2.1}
	With the hypothesis as in Proposition \ref{Conj 2.1 HJNS}, 
	\cite[Conjecture 2.1]{HJNS24-1} is true for $r\leqslant e$.
\end{corollary}

\section{Negative curves on Blowups of Ruled Surfaces}

\begin{theorem}\label{wbnc-hirz}
	Let $P=\{p_1,\dots,p_r\}\subset \mb F_e$ be set of $r$ 
	points. Recall that $\pi_P: \mathbb{F}_{e,P} \to \mathbb{F}_e$ 
	is the blowup of $\mathbb{F}_e$ at the points in $P$. Let $C$ be an integral curve on 
	$\mathbb{F}_{e,P}$ such that $C^2<0$. Then either $C$ is an exceptional divisor or 
	$$C^2\geqslant {\rm min}\{-2,-e-r \}+ 
	\left(e+2-\left\lfloor \frac{r+e}{2}\right\rfloor\right)(C \cdot f_e)\,.$$
\end{theorem}
\begin{proof}
	Let $C=aC_e+bf_e-\sum_{i=1}^rm_iE_i$ be an integral curve on $\mathbb{F}_{e,P}$ 
	with $C^2<0$. If $\pi_P(C)$ is a point, then $C$ is forced to be one of the exceptional 
	divisors $E_i$ and so $C^2=-1$. Let us assume that $\pi_P(C)$ is a curve.
	Then $C$ is the strict transform of $\pi_P(C)$, and so we get all the $m_i\geqslant 0$. 
	
	Let $C^2=-k$ and $K_{e,P}\cdot C=j$. Using  $K_{e,P}\cdot C=j$ we get
	\begin{equation}\label{eqn(-1)}
		2a+2b-ae+j=\sum_{i=1}^rm_i.
	\end{equation}
	
	Note that $h^0(\mb F_e, C_e+\lambda f_e)=2+2\lambda-e$, see Proposition \ref{dim}. Letting 
	$\lambda:=\lfloor \frac{r+e}{2}\rfloor$ we see that $2+2\lambda -e>r$. Thus, there 
	exists an effective divisor $D\in |C_e+\lambda f_e|$ such that $D$ passes through all 
	points in $P$. As $\pi_P(C)$ is integral,  $D$ and $\pi_P(C)$ can have a common component 
	only if $\pi_P(C)$ is a component of $D$. 
	
	First consider the case when $a>1$. 
	As $\pi_P(C)\equiv aC_e+bf_e$, hence, 
	$D-\pi_P(C)\equiv (1-a)C_e+(\lambda-b)f_e$. Since $a>1$, 
	we see that $D-\pi_P(C)$ cannot be effective. So $D$ and $\pi_P(C)$ do not 
	have a common component.  Hence we have 
	$$D\cdot \pi_P(C)\geqslant \sum_{p\in P}(\text{mult}_pD)(\text{mult}_p\pi_P(C)).$$
	Since $\text{mult}_pD\geqslant 1$, we have, 
	$$b+a(\lambda -e)\geqslant \sum_{i=1}^rm_i\,.$$
	Combining this and \eqref{eqn(-1)}, we have 
	$$2a+2b-ae+j=\sum_{i=1}^rm_i \leqslant b+a(\lambda -e)\,,$$
	that is, $b-a(\lambda -2)+j\leqslant 0$. Since $b\geqslant ae$ (use the fact that $\pi_P(C)$ 
	is not $C_e$ or $f_e$ and \cite[Chapter 5, Corollary 2.18(b)]{Ha}), 
	we have $-j\geqslant a(e+2-\lambda )$.
	
	Now by the genus formula we have $$j-k=K_{e,P}\cdot C+C^2\geqslant -2\,,$$
	which yields 
	$$C^2=-k\geqslant -2+a(e+2-\lambda)= -2+(e+2-\lambda)(C \cdot f_e)\,.$$

	Next consider the case $a=1$. Then $C^2=2b-e-\sum_{i=1}^rm_i^2$ and 
	$0\leqslant m_i\leqslant 1$ for 
	all $1\leqslant i \leqslant r$. Thus, using the fact that $b\geqslant 0$,
	$$C^2\geqslant 2b-e-r\geqslant -e-r\,.$$
	
	Finally consider the case $a=0$. In this case $\pi_P(C)$ is forced to be a 
	fiber. Thus, in this case we have $C^2\geqslant -r$. 
	
	This completes the proof of the Theorem.
\end{proof}

\begin{theorem}\label{wbnc-ruled}
	Let $\phi:X=\mb P(E)\to \Gamma$ be a ruled surface over a smooth curve $\Gamma$ of genus $g$,
	with invariant $e$.
	Let $P=\{p_1,\dots,p_r\}\subset X$ be $r$ distinct 
	points in $X$. Let $\pi_P: X_P \to X$ 
	be the blowup of $X$ at the points in $P$. Let $C$ be an integral curve on 
	$X_P$ such that $C^2<0$. Let
	$$\lambda:={\rm max}\left\{2g-1,2g-1+e,g+\left\lfloor \frac{r+e}{2}\right\rfloor\right\}\,.$$ 
	Then either $C$ is an exceptional divisor or 
	$$C^2\geqslant {\rm min}\{-2,-r \}+ 
	\left(e+2-\lambda-2g\right)(C \cdot f_e)\,.$$
\end{theorem}
\begin{proof}
	The proof is similar to the proof of the previous Theorem, and so we will omit
	some details. 
	
	Recall that the Picard group of $X_P$ is generated by the Picard group
	of $\Gamma$, the divisor $C_e$, and the exceptional divisors $E_i$. 
	Let $C=aC_e+\pi^*B-\sum_{i=1}^rm_iE_i$ be an integral curve on $X_{P}$ 
	with $C^2<0$. If $\pi_P(C)$ is a point, then $C$ is forced to be one 
	of the exceptional divisors $E_i$ and so $C^2=-1$. 
	Let us assume that $\pi_P(C)$ is a curve.
	Then $C$ is the strict transform of $\pi_P(C)$, and so we get all 
	the $m_i\geqslant 0$. 
	Let $C^2=-k$ and $K_{X_P}\cdot C=j$. Using  $K_{X_P}\cdot C=j$ we get
	\begin{equation}\label{eqn(-2)}
		2a+2b-ae-2ag+j=\sum_{i=1}^rm_i.
	\end{equation}
	
	Note that $h^0(X, C_e+\pi^*B')=h^0(\Gamma,E\otimes\mc O_\Gamma(B'))$. 
	From the short exact sequence 
	$$0\to \mc O_\Gamma\to E\to L\to 0\,,$$
	it follows that if ${\rm deg}(\mc O_\Gamma(B'))>2g-2$ and 
	${\rm deg}(L\otimes \mc O_\Gamma(B'))>2g-2$ then 
	$$h^0(X,C_e+\pi^*B')=h^0(\Gamma, E\otimes \mc O_\Gamma(B'))=
	2({\rm deg}(B') + 1-g)-e\,.$$
	Letting 
	$$\lambda:={\rm max}\left\{2g-1,2g-1+e,g+\left\lfloor \frac{r+e}{2}\right\rfloor\right\}$$ 
	we see 
	that $2(\lambda + 1-g)-e>r$. Thus, if $B'$ is a divisor of degree $\lambda$ 
	then there exists an effective divisor 
	$D\in |C_e+\pi^*B'|$ such that $D$ 
	passes through all points in $P$. As $\pi_P(C)$ is integral,  $D$ and 
	$\pi_P(C)$ can have a common component 
	only if $\pi_P(C)$ is a component of $D$. 
	
	First consider the case when $a>1$. Write the numerical class of $\pi_P(C)$ 
	as $aC_e+bf_e$, and the numerical class of $D$ as $C_e+\lambda f_e$. Hence, 
	$D-\pi_P(C)\equiv (1-a)C_e+(\lambda-b)f_e$. Since $a>1$, we see that 
	$D-\pi_P(C)$ cannot be effective (as intersecting with $f_e$, a nef class, 
	will give a negative number). So $D$ 
	and $\pi_P(C)$ do not have a common component.  Hence we have 
	$$D\cdot \pi_P(C)\geqslant \sum_{p\in P}(\text{mult}_pD)(\text{mult}_p\pi_P(C)).$$
	Since $\text{mult}_pD\geqslant 1$, we have, 
	$$b+a(\lambda -e)\geqslant \sum_{i=1}^rm_i\,.$$
	Combining this and \eqref{eqn(-2)}, we have 
	$$2a+2b-ae-2ag+j=\sum_{i=1}^rm_i \leqslant b+a(\lambda -e)\,,$$
	that is, $b-a(\lambda -2+2g)+j\leqslant 0$. 
	As $a>1$, using \cite[Chapter 5, 
	Proposition 2.20, Proposition 2.21]{Ha}
	we get $b\geqslant ae$. Thus, we have $-j\geqslant a(e+2-\lambda -2g)$.
	
	Now by the genus formula we have $$j-k=K_{X_P}\cdot C+C^2\geqslant -2\,,$$
	which yields 
	$$C^2=-k\geqslant -2+a(e+2-\lambda-2g)= -2+(e+2-\lambda-2g)(C \cdot f_e)\,.$$

	Next consider the case $a=1$. Then $C^2=2b-e-\sum_{i=1}^rm_i^2$ and 
	$0\leqslant m_i\leqslant 1$ for 
	all $1\leqslant i \leqslant r$. 
	Again, using \cite[Chapter 5, 
	Proposition 2.20, Proposition 2.21]{Ha}
	we see that if $e\geqslant 0$ then $b\geqslant e$
	and if $e<0$ then $b\geqslant 0$. In both cases we have 
	$2b-e\geqslant 0$. Thus, 
	$$C^2\geqslant (2b-e)-r\geqslant -r\,.$$
	Finally, in the case $a=0$, we easily see that $C^2\geqslant -r$.
	This completes the proof of the Theorem. 
\end{proof}

%	\bibliographystyle{halpha}
%	\bibliography{SCC.bib}

\newcommand{\etalchar}[1]{$^{#1}$}

\end{document}